\documentclass[12pt]{amsart}
\usepackage{a4wide}
\usepackage[utf8]{inputenc}
\usepackage[T1]{fontenc} 
\usepackage[english]{babel}
\usepackage{amsmath}
\usepackage{amsfonts}
\usepackage{amssymb}
\usepackage{amsthm}
\usepackage{graphicx}
\usepackage{subcaption}
\usepackage{enumerate}
\usepackage{color}
\usepackage{todonotes}
\usepackage{bbm}
\usepackage{verbatim}
\usepackage{url}

\renewcommand{\d}{\,\mathrm{d}}
\newcommand{\E}[1]{\mathbb{E}\left[#1\right]} 
\newcommand{\no}[1]{\Vert#1\Vert} 
\newcommand{\nos}[1]{\Vert#1\Vert^2} 
\newcommand{\be}[1]{\vert#1\vert} 
\newcommand{\bes}[1]{\vert#1\vert^2} 
\renewcommand{\d}{\,\mathrm{d}}
\renewcommand{\E}[1]{\mathbb{E}\left[#1\right]} 
\renewcommand{\no}[1]{\Vert#1\Vert} 
\renewcommand{\nos}[1]{\Vert#1\Vert^2} 
\renewcommand{\be}[1]{\vert#1\vert} 
\renewcommand{\bes}[1]{\vert#1\vert^2} 

\newcommand{\fe}[1]{\frac{\nabla #1}{ \sqrt{ \vert \nabla #1 \vert^2 +\epsilon^2}}} 

\newcommand{\R}{\mathbb{R}}

\newcommand{\N}{\mathbb{N}}

\newcommand{\into}{\int\limits_{\mathcal{O}}}
\renewcommand{\epsilon}{\varepsilon}
\newcommand{\eps}{\epsilon}
\newcommand{\ska}[1]{\left( #1 \right)} 
\renewcommand{\div}{\mathrm{div}}

\newcommand{\intt}{\int_0^t}

\newcommand{\F}{\mathcal{F}}

\renewcommand{\Xi}{X_{\epsilon,h}^{i}}

\newcommand{\Xmin}{X_{\epsilon,h}^{i-1}}

\newcommand{\Yi}{X_{\eps,\delta,n,h}^{i}} 
\newcommand{\Ymin}{X_{\eps,\delta,n,h}^{i-1}} 
\newcommand{\Zi}{Z^{i}_{\epsilon,h}}
\newcommand{\Zmin}{Z^{i-1}_{\epsilon,h}}

\newcommand{\Hz}{\mathbb{H}^1_0}
\renewcommand{\L}{\mathbb{L}^2}
\newcommand{\Hm}{\mathbb{H}^{-1}}
\renewcommand{\phi}{\varphi}
\renewcommand{\F}{\mathcal{F}}
\renewcommand{\P}{\mathbb{P}}

\newcommand{\vh}{v_h}

\renewcommand{\O}{\mathcal{O}}
\newcommand{\D}{\O}

\newcommand{\Xdn}{X^{\epsilon,\delta}_{n}}
\newcommand{\Xdno}{X^{\epsilon_1,\delta}_{n}}
\newcommand{\Xdnt}{X^{\epsilon_2,\delta}_{n}}
\newcommand{\Xdnu}{X^{\epsilon,\delta_1}_{n_1}}
\newcommand{\Xdnd}{X^{\epsilon,\delta_2}_{n_2}}
\newcommand{\Xdon}{X^{\epsilon,\delta_1}_{n}}
\newcommand{\Xdtn}{X^{\epsilon,\delta_2}_{n}}
\newcommand{\Xdnm}{X^{\epsilon,\delta}_{n,m}}
\newcommand{\Xe}{X^{\epsilon}}

\newcommand{\Xed}{X^{\epsilon,\delta}_n}

\newcommand{\Yc}{\overline{X}_{\tau,h}^{\epsilon,\delta,n}}
\newcommand{\Ycm}{\overline{X}_{\tau_-,h}^{\epsilon,\delta,n}}
\newcommand{\Xc}{\overline{X}_{\tau,h}^{\epsilon}}

\renewcommand{\F}{\mathcal{F}}
\renewcommand{\P}{\mathbb{P}}

\renewcommand{\O}{\mathcal{O}}

\newcommand{\Aed}{A^{\epsilon,\delta}}
\newcommand{\Je}{\mathcal{J}_{\epsilon,\lambda}}
\newcommand{\J}{\mathcal{J}_{\lambda}}
\newcommand{\Jeps}{\mathcal{J}_\eps}
\newcommand{\weak}{\rightharpoonup}

\newtheorem{thms}{Theorem}[section]
\newtheorem{defs}{Definition}[section]
\newtheorem{cors}{Corollary}[section]

\newtheorem{props}{Proposition}[section]
\newtheorem{bems}{Remark}[section]
\newtheorem{lems}{Lemma}[section]

\newcommand{\xx}{x}
\newcommand{\Pm}{\mathcal{P}_m}

\begin{document}
\title[Numerical approximation of the stochastic TV flow]{Convergent numerical approximation of the stochastic total variation flow}

\author{\v{L}ubom\'{i}r Ba\v{n}as}
\address{Department of Mathematics, Bielefeld University, 33501 Bielefeld, Germany}
\email{banas@math.uni-bielefeld.de}
\author{Michael R\"ockner}
\address{Department of Mathematics, Bielefeld University, 33501 Bielefeld, Germany}
\email{roeckner@math.uni-bielefeld.de}
\author{Andr\'e Wilke}
\address{Department of Mathematics, Bielefeld University, 33501 Bielefeld, Germany}
 \email{ awilke@math.uni-bielefeld.de}

\begin{abstract}
We study the stochastic total variation flow (STVF) equation with linear multiplicative noise.
By considering a limit of a sequence of regularized stochastic gradient flows with 
respect to a regularization parameter $\eps$ we obtain the existence of a unique variational solution of the STVF equation which
satisfies a stochastic variational inequality.
We propose an energy preserving fully discrete finite element approximation for the regularized gradient flow equation
and show that the numerical solution converges to the solution of the unregularized STVF equation.
We perform numerical experiments to demonstrate the practicability of the proposed numerical approximation.

{\bf This paper contains a mistake: in the proof of Lemma~\ref{Lemma_Discret_Laplace} the last inequality is not valid. Meanwhile, this mistake has been
fixed in \cite{stvf_erratum} for a slightly modified numerical approximation in spatial dimension $d=1$.
For $d\geq 1$ the validity of the estimate in Lemma~\ref{Lemma_Discret_Laplace} is still open but
the convergence of the numerical approximation can be shown by a different approach, see \cite{stvf_hd}.}
\end{abstract}

\maketitle

\section{Introduction}
We study numerical approximation of the stochastic total variation flow (STVF) 
\begin{align}\label{TVF}
\d X&= \div\left(\frac{\nabla X}{\be{\nabla X}}\right) \d t -\lambda (X - g) \d t +X\d W, &&\text{in } (0,T)\times \O, \nonumber\\
X & = 0 && \text{on } (0,T)\times \partial \O, \\
X(0)&=x_0 &&\text{in } \O, \nonumber
\end{align}
where $\O \subset \R^d$, $d\geq 1$ is a bounded, convex domain with a piecewise $C^{2}$-smooth boundary $\partial \O$,
and $\lambda \geq 0$, $T>0$ are constants. We assume that $x_0,\, g \in \L$ and consider
a one dimensional real-valued Wiener process $W$, for simplicity; generalization for a sufficiently regular trace-class noise is straightforward.

Equation (\ref{TVF}) can be interpreted as a stochastically perturbed gradient flow of the penalized total variation energy functional
\begin{align}\label{jfunc}
\J(u):= \into \be{\nabla u} \d \xx +\frac{\lambda}{2} \into \bes{u-g} \d \xx.
\end{align}
The minimization of above functional, so-called ROF-method, is 
a prototypical approach for image denoising, cf. \cite{ROF};
in this context the function $g$ represents a given noisy image and $\lambda$  serves as a penalization parameter. 
Further applications of the functional include, for instance, elastoplasticity and the modeling of damage and fracture,
for more details see for instance \cite{bm16} and the references therein.

The use of stochastically perturbed gradient flows has proven useful in image processing.
Stochastic numerical methods for models with nonconvex energy functionals 
are able to avoid local energy minima and thus achieve faster convergence and/or 
more accurate results than their deterministic counterparts;
see \cite{kp06} which applies stochastic level-set method in image segmentation, 
and \cite{swp14} which uses stochastic gradient flow of a modified (non-convex) total variation energy functional for binary tomography.

Due to the singular character of total variation flow (\ref{TVF}), it is convenient to perform numerical simulations using a regularized problem
\begin{align}\label{reg.TVF}
 \d \Xe&=  \div\left(\frac{\nabla \Xe}{\sqrt{|\nabla \Xe |^2+\epsilon^2}}\right)\d t-\lambda(\Xe-g)\d t+\Xe\d W &&\text{in } (0,T)\times \O, \nonumber\\
 \Xe & = 0 && \text{on } (0,T)\times \partial \O, \\
 \Xe(0)&=x_0 &&\text{in } \O\,,  \nonumber
\end{align}
with a regularization parameter $\epsilon >0$. In the deterministic setting ($W\equiv 0$) equation \eqref{reg.TVF} corresponds to the gradient flow of the regularized energy functional
\begin{align}\label{def_jepslam}
\Je(u):= \into \sqrt{\bes{\nabla u} +\epsilon^2} \d \xx + \frac{\lambda}{2} \into \bes{u-g} \d \xx.
\end{align}  
It is well-known that the minimizers of the above regularized energy functional converge
to the minimizers of (\ref{jfunc}) for $\epsilon \rightarrow 0$, cf. \cite{Prohl_TVF_numerics} and the references therein.

{
Owing to the singular character of the diffusion term in (\ref{TVF})
the classical variational approach for the analysis of stochastic partial differential equations (SPDEs), see e.g. \cite{kr_07}, \cite{Roeckner_book}, is not applicable to this problem.  
To study well-posedeness of singular gradient flow problems it is convenient to apply the solution framework developed in \cite{Roeckner_TVF_paper} 
which characterizes the solutions of (\ref{TVF}) as stochastic variational inequalities (SVIs).
In this paper, we show the well posedness of SVI solutions using the practically relevant regularization procedure (\ref{reg.TVF})
which, in the regularization limit, yields a SVI solution in the sense of \cite{Roeckner_TVF_paper}.
Throughout the paper, we will refer to the solutions which satisfy a stochastic variational inequality 
as SVI solutions, and to the classical SPDE solutions as variational solutions. 
Convergence of numerical approximation of (\ref{reg.TVF}) in the deterministic setting ($W\equiv 0$)
has been shown in \cite{Prohl_TVF_numerics}. 
 Analogically to the deterministic setting, we construct an implementable finite element approximation of the problem (\ref{TVF})
via the numerical discretization of the regularized problem (\ref{reg.TVF}). The scheme is implicit in time
and preserves the gradient structure of the problem, i.e., it satisfies a discrete energy inequality.
The deterministic variational inequality framework
used in the the numerical analysis of \cite{Prohl_TVF_numerics} is not directly transferable to the stochastic setting.
{Instead, we show the convergence of the proposed numerical approximation of (\ref{reg.TVF}) to the SVI solution of (\ref{TVF}) via an additional
regularization step on the discrete level. The convergence analysis of the discrete approximation is inspired by the analytical approach of 
\cite{Gess_Stability} where the SVI solution concept was applied to the stochastic $p$-Laplace equation.} 
As far as we are aware, the present work is the first to show convergence of implementable numerical approximation for singular stochastic gradient flows
in the framework of stochastic variational inequalities.

The paper is organized as follows. In Section~\ref{sec_not} we introduce the notation and state some auxiliary results.
The existence of a unique SVI solution of the regularized  stochastic TV flow (\ref{reg.TVF})
and its convergence towards a unique SVI solution of (\ref{TVF}) for $\eps\rightarrow 0$
is shown in Section~\ref{sec_exist}.
In Section~\ref{sec_num} we introduce a fully discrete finite element scheme for the regularizared problem (\ref{reg.TVF})
and show its convergence to the SVI solution of (\ref{TVF}).
Numerical experiments are presented in Section~\ref{sec_sim}.

\section{Notation and preliminaries}\label{sec_not}
Throughout the paper we denote by $C$ a generic positive constant that may change from line to line. 
For $1\leq p \leq \infty  $, we denote by $(\mathbb{L}^p,\no{\cdot}_{\mathbb{L}^p})$ the standard spaces of $p$-th order integrable functions on $\O$,
and use $\no{\cdot}:= \no{\cdot}_{\L}$ and $(\cdot,\cdot):=(\cdot,\cdot)_{\L}$ for the $\L$-inner product. 
For $k \in \N$ we denote the usual Sobolev space  on $\O$ as $(\mathbb{H}^k,\no{\cdot}_{\mathbb{H}^k})$,
and $(\Hz,\no{\cdot}_{\Hz})$ stands for the $\mathbb{H}^1$ space with zero trace on $\partial \O$ with its dual space $(\Hm,\no{\cdot}_{\Hm})$. 
Furthermore, we set $\langle \cdot ,\cdot \rangle:=\langle \cdot ,\cdot \rangle_{\Hm \times \Hz}$, where $\langle \cdot ,\cdot \rangle_{\Hm \times \Hz}$ 
is the duality pairing between $\Hz$ and $\Hm$.  The functional (\ref{def_jepslam}) with 
$\lambda=0$ will be denoted as $\mathcal{J}_\eps := \mathcal{J}_{\eps,0}$.
We say that a {mapping $X: \Omega\times(0,T) \rightarrow \L$}
is $\F_t$-progressively measurable if $X \mathbbm{1}_{[0,t]}$ is $\F_t\otimes \mathcal{B}([0,t])$-measurable for all $t \in [0,T]$.

For the convenience of the reader we state some basic definitions below.
\begin{defs}\label{resolvent and Yosida approximation}
Let $\mathbb{H}$ be a real Banach space, $A :D(A)\rightarrow \mathbb{H}$ a linear operator and $\rho(A)$ its resolvent set.
For a real number $\xi \in \rho(A)$ we define the resolvent $R_{\xi}: \mathbb{H} \rightarrow \mathbb{H}$ of $A$ as 
\begin{align*}
R_{\xi}(x):=( I-\xi A)^{-1}x\,.
\end{align*}
Furthermore we define the Yosida approximation of $A$ as
\begin{align*}
T_{\xi}(x):=AR_{\xi}=\frac{1}{\xi}(I- R_{\xi})x.
\end{align*}
\end{defs}

\begin{defs}\label{orth.Proj}
The mapping $\mathcal{P}_m : \L \rightarrow \mathbb{V}_m\subset\L$ which satisfies
\begin{align*}
\ska{w-\mathcal{P}_m w,v_m}=0 ~~ \forall v_m \in \mathbb{V}_m.
\end{align*}
defines the $\L$-orthogonal projection onto $\mathbb{V}_m$.
\end{defs} 

\begin{defs}\label{Bounded Variation} 
 A function $u \in L^1(\O)$ is called a function of bounded variation, if its total variation
 \begin{align}\label{total variation}
 \into \be{\nabla u}\d\xx  := \sup\left\{-\into u\, \div\, \mathbf{v} \d\xx;~ \mathbf{v} \in C^{\infty}_0(\O,\R^d), ~\no{\mathbf{v}}_{L^{\infty}}\leq 1\right\},
 \end{align}
 is finite. The space of functions of bounded variations is denoted by $BV(\O)$.

For $u \in BV(\O)$ we denote
$$
 \into \sqrt{\be{\nabla u}^2 + \eps^2}\d\xx  := \sup\left\{\into \Big(-u\, \div\, \mathbf{v} + \eps\sqrt{1-|\mathbf{v}(\xx)|^2}\Big)\d\xx;~ \mathbf{v} \in C^{\infty}_0(\O,\R^d), ~\no{\mathbf{v}}_{L^{\infty}}\leq 1\right\}\,.
$$
\end{defs}

The following proposition plays an important role in the analysis below; the proposition
holds for convex domains with piecewise smooth boundary, which includes the case of practically relevant
polygonal domains, cf. \cite[Proposition 8.2 and Remark 8.1]{Roeckner_TVF_paper}.
\begin{props}\label{Resolvent_estimate}
Let $\mathcal{O}\subset\mathbb{R}^d$, $d\geq 1$ be a bounded domain with a piecewise $C^2$-smooth and convex boundary.
Let $g :[0,\infty) \rightarrow [0,\infty)$ be a continuous  and convex function of at most quadratic growth such that $g(0)=0$,
then it holds 
\begin{align}
\into g (\be{\nabla R_{\xi}(y)}) \d \xx \leq \into g (\be{\nabla y}) \d \xx, ~~\forall y \in \Hz.
\end{align}
\end{props}

\section{Well posedness of STVF}\label{sec_exist}
In this section we show existence and uniques of the SVI solution of (\ref{TVF}) (see below for a precise definition)
via a two-level regularization procedure. 
 To be able to treat problems with $\L$-regular data, i.e., $x_0 \in L^2(\Omega,\F_0;\L)$, $g\in \L$
we consider a $\Hz$-approximating sequence {$\{x^n_0\}_{n\in \mathbb{N}} \subset L^2(\Omega,\F_0;\Hz)$
s.t.  $x^n_0 \rightarrow x_0$
in $L^2(\Omega,\F_0; \L)$ for $n\rightarrow \infty$ and $\{g^n\}_{n\in \mathbb{N}} \subset \Hz$
s.t.  $g^n \rightarrow g$
in $\L$ for $n\rightarrow \infty$.} We
introduce a regularization of (\ref{reg.TVF}) as
\begin{align}\label{vis.TVF}
 \d \Xdn=&\delta \Delta \Xdn \d t +\div\left(\frac{\nabla \Xdn}{\sqrt{|\nabla \Xdn |^2+\epsilon^2}}\right)\d t \nonumber\\
 &-\lambda(\Xdn-g^n)\d t+\Xdn\d W(t)  &&\text{in }(0,T)\times \O,\\
\Xdn(0)=&x^n_0 &&\text{in } \O,\nonumber
\end{align}
where $\delta>0$ is an additional regularization parameter.

We define the operator $\Aed : \Hz \rightarrow \Hm$ as
\begin{align}\label{Operator}
{\langle\Aed u, v\rangle_{\Hm \times \Hz}=\into \delta \nabla u \nabla v +\fe{u}\nabla v +\lambda (u-   g^n)v \d\xx  ~~\forall  u,v \in \Hz },
\end{align}
and note that \eqref{vis.TVF} is equivalent to
\begin{align}
\d \Xdn +\Aed \Xdn \d t&=\Xdn \d W(t)\,,\\
\Xdn(0)&=x^n_0. \nonumber
\end{align}

The operator $\Aed$ is coercive, demicontinuos and satisfies (cf. \cite[Remark 4.1.1]{Roeckner_book})
\begin{align}
&\langle\Aed(u)-\Aed(v),u-v\rangle_{\Hm \times \Hz}\geq \delta \|\nabla (u-v)\|^2 + \lambda\nos{u-v}, && \forall u,v \in \Hz,\label{Monotonicity}\\
& \no{\Aed(u)}_{\Hm} \leq C(\delta,\lambda,\|g^n\|)(\no{u}_{\Hz}+1), && \forall u \in \Hz.\label{a_bnd}
\end{align}
The following monotonicity property,
which follows from the convexity of the function $\sqrt{\bes{\cdot}+\epsilon^2}$, will be used frequently in the subsequent arguments
\begin{align}\label{eps.convexity.inequality}
&\ska{\fe{X}-\fe{Y},\nabla(X-Y)}
\nonumber\\
=&\ska{\fe{X},\nabla(X-Y)}+\ska{\fe{Y},\nabla(Y-X)}
\\
&\geq \Jeps(X)-\Jeps(Y)+\Jeps(Y)-\Jeps(X)=0.
\nonumber
\end{align}

The existence and uniqueness of a variational solution of {(\ref{vis.TVF})} 
is established in the next lemma; we note that the result only requires $\L$-regularity
of data.
\begin{lems}\label{lem_visc_exist}
For any $\eps, \delta>0$ and $x^n_0 \in L^2(\Omega,\F_0;\Hz)$,  $g^n\in \Hz$ 
there exists a unique variational solution $\Xdn \in L^2(\Omega;C([0,T];\L))$ 
of \eqref{vis.TVF}. Furthermore, there exists a $C\equiv C(T)>0$ such that the following estimate holds
\begin{align*}
\E{\sup\limits_{t \in [0,T]} \nos{\Xdn(t)}} \leq C (\E{\nos{x_0^n}} + \|g^n\|^2).
\end{align*}
\end{lems}
\begin{proof}[\textbf{Proof of Lemma \ref{lem_visc_exist}}]
On noting the properties (\ref{Monotonicity})-(\ref{a_bnd}) of the operator $\Aed$ for $\eps,\delta>0$
the classical theory, cf. \cite{Roeckner_book},
implies that for any given data $x^n_0 \in L^2(\Omega,\F_0;\Hz)$,  $g^n\in \Hz$ there exists
a unique variational solution $\Xdn \in L^2(\Omega;C([0,T];\L))$ 
of (\ref{vis.TVF}) which satisfies the stability estimate.
\end{proof}
In next step, we  show a priori estimate for the solution of (\ref{vis.TVF}) in stronger norms;
the estimate requires $\Hz$-regularity of the data.
\begin{lems}\label{laplace_energy_estimate}
Let $x_0^n\, \in L^2(\Omega,\F_0;\Hz)$, $g^n \in \Hz$.
There exists a constant $C\equiv C(T)$
such that for any $\eps,\, \delta>0$ the corresponding variational solution $\Xdn$ of \eqref{vis.TVF}  satisfies
\begin{align}
\E{\sup\limits_{t \in [0,T]} \nos{\nabla \Xdn(t)}}+\delta \E{\intt \nos{\Delta \Xdn(s)} \d s} \leq C\left(\E{\nos{x_0^n}_{\Hz}}+\nos{g^n}_{\Hz}\right).
\end{align}
\end{lems}
\begin{proof}[\textbf{Proof of Lemma \ref{laplace_energy_estimate}}]
Let $\{e_i\}_{i=0}^{\infty}$ be an orthonormal basis of eigenfunctions of the Dirichlet Laplacian $-\Delta$ on $\L$
and $\mathbb{V}_m := \text{span}\{e_0,\ldots, e_m\}$. Let $\Pm :\L \rightarrow \mathbb{V}_m$
be the $\L$-orthogonal projection onto $\mathbb{V}_m$. 

For fixed $\eps,\,\delta,\,n$ the Galerkin approximation $\Xdnm \in \mathbb{V}_m$ of $\Xdn$ satisfies
\begin{align}\label{gal.vis.TVF}
\d \Xdnm =& \delta \Delta \Xdnm \d t +\Pm \div\left(\frac{\nabla \Xdnm}{\sqrt{|\nabla \Xdnm |^2+\epsilon^2}}\right)\d t
\nonumber \\
 & -\lambda  (\Xdnm-{\Pm g^n})\d t+ \Xdnm \d W(t),\\
\Xdnm(0)=&\Pm  x^n_0. \nonumber
\end{align}
By standard arguments, cf. \cite[{Theorem 5.2.6}]{Roeckner_book},
there exists a $\Xdn \in L^2(\Omega;C([0,T];\L))$ such that $\Xdnm \weak \Xdn $ in $L^2(\Omega\times(0,T);\L)$
for $m\rightarrow \infty$. 
We use It\^o's formula for  $\nos{\nabla \Xdnm(t)}$ 
to obtain
\begin{align}\label{ito_nabla}
\nonumber
\frac{1}{2} \nos{\nabla \Xdnm(t)}=& \frac{1}{2}\nos{\nabla \Pm  x_0^n}-\delta \intt \nos{\Delta \Xdnm(s)}\d s
\nonumber 
\\
\nonumber
&-\intt \ska{\div \fe{\Xdnm(s)},\Delta \Xdnm(s)}\d s 
\\
&-\lambda\intt \ska{(\Xdnm(s)-{g^n}),\Delta \Xdnm(s)}\d s
\\ \nonumber
&-\intt \ska{\Delta \Xdnm(s), \Xdnm(s)\d W(s)}\d s
\\
\nonumber &+{\frac{1}{2}}\intt \nos{ \Xdnm(s)}_{\Hz}\d s.  
\end{align}
Let $T_\xi$ be the Yosida-approximation
and $R_{\xi}$ the resolvent of the Dirichlet Laplacian $-\Delta$ on $L^2$, respectively; see Definition~\ref{resolvent and Yosida approximation}.
By the convexity, cf. (\ref{eps.convexity.inequality}), we get
\begin{align*}
&\ska{-\Delta \Xdnm(s),\div \fe{\Xdnm(s)}}\\
&=\lim\limits_{\xi \rightarrow \infty}\ska{T_\xi \Xdnm(s),\div \fe{\Xdnm(s)}}\\
&=\lim\limits_{\xi \rightarrow \infty} \frac{1}{\xi}\ska{\Xdnm(s)- R_{\xi}\Xdnm(s),\div \fe{\Xdnm(s)}}\\
&=\lim\limits_{\xi\rightarrow \infty} \frac{1}{\xi}\ska{\nabla R_{\xi}\Xdnm(s)-\nabla \Xdnm(s),\fe{\Xdnm(s)}}\\
&\leq \lim\limits_{\xi \rightarrow \infty} \frac{1}{\xi} \left( \mathcal{J}_{\epsilon}( R_{\xi}\Xdnm(s)) -\mathcal{J}_{\epsilon}( \Xdnm(s))\right)\\
&\leq 0,
\end{align*}
where we used Proposition \ref{Resolvent_estimate} in the last step above.
{The Burkholder-Davis-Gundy inequaltiy for $p=1$ implies that
\begin{align}\label{BDG nabla}
\E{\sup\limits_{t \in [0,T]} \intt \nos{\nabla \Xdnm(s)} \d W(s)} \leq& C\E{ \left(\int_0^T \no{\nabla \Xdnm(s)}^4 \d s\right)^{\frac{1}{2}}} \nonumber \\
  \leq& C\E{ \sup\limits_{t \in [0,T]}\no{\nabla \Xdnm(t)}\left(\int_0^T \no{\nabla \Xdnm(s)}^2 \d s\right)^{\frac{1}{2}} } \\
\leq& \frac{1}{4} \E{\sup\limits_{t \in [0,T]}\nos{\nabla \Xdnm(t)}}  +C\E{\int_0^T \nos{\nabla \Xdnm(s)} \d s}\,.  \nonumber
\end{align}
After taking supremum over $t$ and expectation in \eqref{ito_nabla}, using \eqref{BDG nabla} along with the Tonelli and Gronwall lemmas we obtain}
\begin{align*}
\E{\sup\limits_{t \in [0,T]} \nos{\nabla \Xdnm(t)} +\delta \int_0^T \nos{\Delta \Xdnm(s)}\d s} \leq C(\E{\nos{x_0^n}_{\Hz}} + \nos{g^n}_{\Hz}).
\end{align*}
Hence, from the sequence $\{\Xdnm\}_{m\in\N}$ we can extract a subsequence (not relabeled), s.t. for $m \rightarrow \infty$
\begin{align*}
&\Xdnm \weak \Xdn ~\text{in}~ L^2(\Omega;L^2((0,T);\mathbb{H}^2)\\
&\Xdnm \weak^* \Xdn ~\text{in}~ L^2(\Omega;L^{\infty}((0,T);\Hz)).
\end{align*}
By lower-semicontinuity of the norms, we get
\begin{align*}
\E{\sup\limits_{t \in [0,T]}\frac{1}{2} \nos{\nabla \Xdn(t)} +\delta \int_0^T \nos{\Delta \Xdn(s)}\d s} \leq C(\E{\nos{x_0^n}_{\Hz}}+\nos{g^n}_{\Hz})  .
\end{align*}
\end{proof}
{We define the following functionals
\begin{align*}
\bar{ \mathcal{J}}_{\eps,\lambda}(u)=
\begin{cases}
\mathcal{J}_{\eps,\lambda}(u) + \int_{\partial \O} \be{\gamma_0(u)} \d \mathcal{H}^{n-1} ,~~ \text{for}~ u \in BV(\O)\cap L^2(\O),\\
+\infty  ,~~ \text{for}~ u \in BV(\O)\setminus L^2(\O)
\end{cases}
\end{align*}
and (for $\eps=0$)
\begin{align*}
\bar{ \mathcal{J}}_\lambda (u)=
\begin{cases}
\mathcal{J}_\lambda (u) + \int_{\partial \O} \be{\gamma_0(u)} \d \mathcal{H}^{n-1} , ~~\text{for}~ u \in BV(\O)\cap L^2(\O),\\
+\infty  , ~~\text{for}~ u \in BV(\O)\setminus L^2(\O),
\end{cases}
\end{align*}
where $\gamma_0(u) $ is the trace of $u$ on the boundary and $\d \mathcal{H}^{n-1}$
is the Hausdorff measure.
$\bar{\mathcal{J}}_{\eps,\lambda}$ and $\bar{\mathcal{J}}_{\lambda}$ are  both convex and lower semicontinuous on $L^2(\O)$ and  the lower semicontinuous hulls of $\bar{\mathcal{J}}_{\eps,\lambda}\vert_{\Hz}$ or $\bar{\mathcal{J}}_{\lambda}\vert_{\Hz}$ respectively, cf. \cite[Proposition 11.3.2]{book_attouch}.}
We define the SVI solution of (\ref{reg.TVF}) and (\ref{TVF})
analogically to \cite[Definition~3.1]{Roeckner_TVF_paper} as a stochastic variational inequality.
\begin{defs}\label{def_varsoleps}
Let $0  < T < \infty$, {$\varepsilon \in [0,1]$} and {$x_0 \in L^2(\Omega,\F_0;\L)$ and $g \in \L$}.
Then an $(\F_t)$-{adapted} stochastic process {$\Xe \in L^2(\Omega; C([0,T];\L))\cap  L^1(\Omega; L^1((0,T);BV(\O)))$ 
(denoted by $X \in L^2(\Omega; C([0,T];\L))\cap  L^1(\Omega; L^1((0,T);BV(\O)))$ for $\eps=0$)}
is called an {SVI  solution} of (\ref{reg.TVF}) (or (\ref{TVF}) if $\varepsilon=0$) if $\Xe(0)=x_0$ ($X(0)=x_0$), and
for each $(\F_t)$-progressively measurable process $G\in L^2(\Omega \times (0,T),\L)  $ and for each $(\F_t)$-adapted $\L$-valued process 
{$Z$} with $\P$-a.s. continuous sample paths, {s.t. $Z \in L^2(\Omega \times (0,T);\Hz)$}, which satisfy the equation 
\begin{align}\label{test}
\d Z(t)= -G(t) \d t +Z(t)\d W(t), ~ t\in[0,T],
\end{align}
it holds for {$\eps \in (0,1]$} that
\begin{align}\label{reg.SVI}
\frac{1}{2}& \E{\nos{\Xe(t)-Z(t)}}+\E{\intt {\bar{\mathcal{J}}_{\eps,\lambda}}(\Xe(s)) \d s} \nonumber\\
&\leq  \frac{1}{2} \E{\nos{x_0-Z(0)}}+\E{\intt { \bar{\mathcal{J}}_{\eps,\lambda}}(Z(s)) \d s}  \\
&+ \frac{1}{2}\E{\intt \nos{\Xe(s)-Z(s)} \d s}
+\E{\intt \ska{\Xe(s)-Z(s),G} \d s}\,,\nonumber
\end{align}
and analogically for $\eps=0$ it holds that
\begin{align}\label{SVIeps0}
\frac{1}{2}& \E{\nos{X(t)-Z(t)}}+\E{\intt { \bar{\mathcal{J}}_{\lambda}}(X(s)) \d s} \nonumber\\
&\leq  \frac{1}{2} \E{\nos{x_0-Z(0)}}+\E{\intt { \bar{\mathcal{J}}_{\lambda}}(Z(s)) \d s}  \\
&+\frac{1}{2} \E{\intt \nos{X(s)-Z(s)} \d s}
+\E{\intt \ska{X(s)-Z(s),G} \d s}\nonumber.
\end{align}
\end{defs}

In the next theorem we show the existence and uniqueness of a SVI solution to \eqref{reg.TVF} for $\varepsilon > 0$ 
in the sense of the Definition~\ref{def_varsoleps}.
\begin{thms}\label{Thm_reg.SVI}
Let $0 < T < \infty$ and $x_0 \in L^2(\Omega,\F_0;\L)$, $ g \in \L$.
For each $\eps\in (0,1]$ there exists a unique SVI  solution $\Xe$ of (\ref{reg.TVF}).
Moreover, any two SVI  solutions $\Xe_1,\Xe_2$ with $x_0\equiv x^1_0$, $g\equiv g^1$  and $x_0\equiv x^2_0$, $g\equiv g^2$ satisfy
\begin{align}\label{reg_stability_inequality}
&\E{\nos{\Xe_1(t)-\Xe_2(t)}}\leq C\left(\E{\nos{x^1_0-x^2_0}}+\nos{g^1-g^2} \right )\, ,
\end{align}
for all $t \in [0,T]$.
\end{thms}
\begin{proof}[\textbf{Proof of Theorem \ref{Thm_reg.SVI}}]\hfill \\
We show that for fixed $\epsilon>0$ the sequence $\{\Xdn\}_{\delta,n}$ of variational solutions of (\ref{vis.TVF})
is a Cauchy-sequence w.r.t. $\delta$ for any fixed $n\in \mathbb{N}$, 
and then show that it is a Cauchy-sequence w.r.t. $n$ for $\delta\equiv0$. 

We denote by $\Xdnu,\Xdnd$ the solutions of \eqref{vis.TVF}  
for $\delta\equiv\delta_1$, $\delta\equiv \delta_2$
and {$x_0\equiv x^{n_1}_0\in L^2(\Omega,\F_0;\Hz)$, $x_0\equiv x^{n_2}_0 \in L^2(\Omega,\F_0;\Hz)$}, respectively,
where $x^{n_1}_0$, $x^{n_2}_0$ belong to the $\Hz$-approximating sequence of $x_0\in L^2(\Omega,\F_0;\L)$. 
By It\^o's formula it follows that
\begin{align*}
\frac{1}{2}& \nos{\Xdnu(t)-\Xdnd(t)}\qquad\\
& = \frac{1}{2}\nos{x^{n_1}_0-x^{n_2}_0} +\intt \ska{\delta_1\Delta \Xdnu(s)-\delta_2\Delta\Xdnd(s),\Xdnu(s)-\Xdnd(s)}\d s
\\
&\qquad -\intt \ska{\frac{\nabla \Xdnu(s)}{ \sqrt{ \vert \nabla \Xdnd(s) \vert^2 +\epsilon^2}}-\frac{\nabla \Xdnd(s)}{ \sqrt{ \vert \nabla \Xdnd(s) \vert^2 +\epsilon^2}},\nabla (\Xdnu(s)-\Xdnd(s))}\d s \\
&\qquad -\lambda\intt \nos{(\Xdnu(s)-\Xdnd(s)}\d s +\intt \nos{\Xdnu(s)-\Xdnd(s)}\d W(s)\\
 &\qquad +\intt \nos{\Xdnu(s)-\Xdnd(s)}\d s.
\end{align*}
We note that
\begin{align*}
&\ska{\delta_1\Delta \Xdnu(s)-\delta_2\Delta\Xdnd(s),\Xdnu(s)-\Xdnd(s)}\\
&=-\ska{\delta_1 \nabla \Xdnu(s)-\delta_2\nabla\Xdnd(s),\nabla\Xdnu(s)-\nabla\Xdnd(s)}\\
&\leq C(\delta_1+\delta_2)(\nos{\nabla \Xdnu(s)}+\nos{\nabla \Xdnd(s)}).
\end{align*}
Hence by using the convexity \eqref{eps.convexity.inequality},  Lemma \ref{laplace_energy_estimate}, the Burkholder-Davis-Gundy inequality for $p=1$,
the Tonelli and Gronwall lemmas we obtain
{
\begin{align}\label{delta_datum_estimate}
& \E{\sup\limits_{t \in [0,T]} \nos{\Xdnu(t)-\Xdnd(t)}}
\leq  C\E{\nos{x_0^{n_1}-x_0^{n_2}}}
\nonumber \\ 
&\qquad +C\Big(\E{\nos{x_0^{n_1}}_{\Hz}},\E{\nos{x_0^{n_2}}_{\Hz}},\nos{g^n}_{\Hz}\Big)(\delta_1+\delta_2).
\end{align}
}


Inequality \eqref{delta_datum_estimate} implies for $x_0^{n_1} \equiv x_0^{n_2}\equiv x_0^n$ that
\begin{align*}
\E{\sup\limits_{t \in [0,T]} \nos{\Xdon(t)-\Xdtn(t)}} 
\leq C\left(\E{\nos{x_0^n}_{\Hz}},\nos{g^n}_{\Hz}\right)(\delta_1+\delta_2).
\end{align*}
Hence for any fixed $n$, $\epsilon$ there exists a $\{\F_t \}$-adapted process $\Xe_n \in L^2(\Omega,C([0,T];\L))$, s.t.	
\begin{align}\label{delta_limit}
\lim_{\delta \rightarrow 0} \E{\sup\limits_{t \in [0,T]}\nos{\Xdn(t)-\Xe_n(t)}} \rightarrow 0.
\end{align}
For fixed $n_1$, $n_2$, $\epsilon$ 
we get from \eqref{delta_datum_estimate} using (\ref{delta_limit})
by the lower-semicontinuity of norms that
\begin{align}\label{lim_n}
&\E{\sup\limits_{t \in [0,T]} \nos{X^{\epsilon}_{n_1}(t)-X^{\epsilon}_{n_2}}}
\leq 
\liminf_{\delta\rightarrow 0}\E{\sup\limits_{t \in [0,T]} \nos{X^{\epsilon,\delta}_{n_1}(t)-X^{\epsilon,\delta}_{n_2}}}
\\  \nonumber & \qquad 
\leq\frac{1}{2}\E{\nos{x_0^{n_1}-x_0^{n_2}}}.
\end{align}
Since $x_0^{n_1},\,x_0^{n_2}\rightarrow x_0$ for $n_1,n_2\rightarrow \infty$
we deduce from (\ref{lim_n}) that for any fixed $\epsilon$
there exists an $\{\F_t \}$-adapted process $\Xe \in L^2(\Omega;C([0,T];\L))$  such that 
\begin{align}\label{lim_nnew}
\lim_{n \rightarrow \infty} \E{\sup\limits_{t \in [0,T]}\nos{\Xe(t)-\Xe_n(t)}} \rightarrow 0.
\end{align}

In the next step, we show that the limiting process $\Xe$ is a SVI solution of \eqref{reg.TVF}.
We subtract the process
\begin{align*}
\d Z(t)= -G(t)\d t + Z(t)\d W(t)\,,
\end{align*} 
with $Z(t)=z_0$ from (\ref{vis.TVF}) and obtain
\begin{align*}
\d \left(\Xdn(t) -Z(t) \right)=  \left( -\Aed \Xdn(t)+G(t) \right)\d t + \left(\Xdn(t)-Z(t)\right)\d W(t).
\end{align*} 
The  It\^o formula implies
\begin{align}\label{ito_eps}
\frac{1}{2}&\E{\nos{\Xdn(t)-Z(t)}}
\nonumber \\
=&\frac{1}{2}\E{\nos{\Xdn(0)-z_0}}-\E{\intt \langle\Aed \Xdn(s),\Xdn(s)-Z(s) \rangle \d s} 
\\
&+\E{\intt \ska{G(s),\Xdn(s)-Z(s)}\d s}
+ \E{\intt \nos{\Xdn(s)-Z(s)} \d s}.
\nonumber 
\end{align}
We rewrite the second term on the right-hand side in above inequality as
\begin{align*}
&\E{\intt \langle\Aed \Xdn(s),\Xdn(s)-Z(s)\rangle \d s}\\
&=\E{\intt \delta(\nabla \Xdn(s),\nabla (\Xdn(s) - Z(s))) \d s}
\\&\quad +\E{\intt \left(\fe{\Xdn(s)},\nabla \big(\Xdn(s)-Z(s)\big)\right)\d s} 
\\ &\quad + \E{\intt \lambda(\Xdn(s)-{g^n},\Xdn(s)-Z(s)) \d s}.
\end{align*}
The convexity of $\mathcal{J}_\epsilon$ along with the Cauchy-Schwarz and Young's inequalities imply that
\begin{align*}
&\E{\intt (\fe{\Xdn(s)},\nabla (\Xdn(s)-Z(s)))\d s} 
\\ & +\E{\intt \lambda(\Xdn(s)-{g^n},\Xdn(s)-Z(s)) \d s}
\\& \geq \E{\intt \mathcal{J}_{\epsilon}(\Xdn(s))-\mathcal{J}_{\epsilon}(Z(s))\d s} 
\\ & +\E{\intt \frac{\lambda}{2}\nos{\Xdn(s)-{g^n}}-\frac{\lambda}{2}\nos{Z(s)-{g^n}} \d s} .
\end{align*}
By combining two inequalities above with (\ref{ito_eps}) we get
\begin{align}\label{est_svidelta1}
&\frac{1}{2}\E{\nos{\Xdn(t)-Z(t)}}+\E{\intt {\mathcal{J}_{\epsilon}(\Xdn(s))\d s}+\frac{\lambda}{2}\nos{\Xdn(s)-{g^n}} \d s}
\nonumber \\
&+ \frac{\delta}{2}\E{\intt \nos{\nabla  \Xdn(s)})\d s} 
\nonumber \\
& \leq \frac{1}{2}\E{\nos{\Xdn(0)-Z(0)}} + \E{\intt {\mathcal{J}_{\epsilon}(Z(s))\d s}+\frac{\lambda}{2}\nos{Z(s)-{g^n}}\d s} 
\\
&\quad + \frac{\delta}{2}\E{\intt \nos{\nabla  Z(s)})\d s}+\E{\intt \ska{G(s),\Xdn(s)-Z(s)}\d s}
\nonumber \\
&\quad +\E{\intt \nos{\Xdn(s)-Z(s)} \d s}. \nonumber
\end{align}
{Since $\Xed \in \Hz$ and $Z \in \Hz$ it holds that $\mathcal{J}_{\eps,\lambda}(\Xed)=\bar{\mathcal{J}}_{\eps,\lambda}(\Xed)$  and $\mathcal{J}_{\eps,\lambda}(Z)=\bar{\mathcal{J}}_{\eps,\lambda}(Z)$}. The lower-semicontinuity of {{$\bar{\mathcal{J}}_{\eps,\lambda}$ in $BV(\O)$}} with respect to convergence in $\mathbb{L}^1$, cf. \cite{Ambrosio_functions_of_BV},
and (\ref{delta_limit}), (\ref{lim_nnew}) {and the strong convergence  $g^n \rightarrow g $ in $\L$} 
imply that for $\delta  \rightarrow 0$ and $n \rightarrow \infty$
the limiting process $\Xe \in L^2(\Omega; C([0,T];\L)$ satisfies \eqref{reg.SVI}.

To conclude that $\Xe$ is a SVI solution of (\ref{reg.TVF}) it remains to show that $\Xe \in L^1(\Omega; L^1((0,T);BV(\O)))$.
Setting $G\equiv 0$ in \eqref{test} (which implies $Z\equiv 0$ by (\ref{test})) yields
\begin{align}\label{cont_ener}
\frac{1}{2}& \E{\nos{\Xe(t)}}+\E{\intt {\bar{\mathcal{J}}_{\eps,\lambda}}(\Xe(s)) \d s} \nonumber\\
&\leq  \frac{1}{2} \E{\nos{x_0}}+\E{\intt {\bar{\mathcal{J}}_{\eps,\lambda}}(0) \d s}  
+ \E{\intt \nos{\Xe(s)} \d s}\,.
\end{align}
On noting that (cf. Definition~\ref{Bounded Variation} or \cite[proof of Theorem 1.3]{Prohl_TVF_numerics})
\begin{align*}
{\bar{\mathcal{J}}_{\eps,\lambda}}(\Xe)
\geq& {\bar{\mathcal{J}}_{\lambda}}(\Xe),
\end{align*}
and 
${\bar{\mathcal{J}}_{\eps,\lambda}}(0)=\epsilon\be{\O} + \frac{\lambda}{2}\nos{g}$,
we deduce from (\ref{cont_ener}) that
\begin{align*}
\frac{1}{2}& \E{\nos{\Xe(t)}}+\E{\intt {\bar{\mathcal{J}}_{\lambda}}(\Xe(s)) \d s} \nonumber\\
&\leq  \frac{1}{2} \E{\nos{x_0}}+\E{\intt\epsilon\Big(\be{\O} + \frac{\lambda}{2}\nos{g}\Big)  \d s}  
+ \E{\intt \nos{\Xe(s)} \d s}\,.\nonumber
\end{align*}
Hence, by the Tonelli and Gronwall lemmas it follows that
\begin{align}\label{BV-eps-esti.}
\frac{1}{2} \E{\nos{\Xe(t)}}&+\E{\intt {\bar{\mathcal{J}}_{\lambda}}(\Xe(s)) \d s} \nonumber \\
&\leq C_T\exp(T)\left(\E{\nos{x_0}}+ \be{\O} + \lambda\nos{g}  \right).
\end{align}
Hence $\Xe \in L^2(\Omega; C([0,T];\L))\cap  L^1(\Omega; L^1((0,T);BV(\O)))$ is a SVI solution of (\ref{reg.TVF}) for $\epsilon \in (0,1]$.

In the next step we show the uniqueness of the SVI solution. 
Let $X^{\epsilon}_1, X^{\epsilon}_2$ be two SVI solutions to \eqref{reg.TVF} 
for a fixed  $\eps\in(0,1]$
with initial values 
$x_0\equiv x^1_0, x^2_0$ and $g \equiv g^1,g^2$, respectively. Let $\{x^{2,n}_0\}_{n\in\mathbb{N}} \subset L^2(\Omega,\F_0;\Hz)$ be a sequence,   
s.t. $x^{2,n}_0 \rightarrow x^2_0$ in $L^2(\Omega,\F_0;\L)$ and {$\{g^{2,n}\}_{n\in\mathbb{N}} \subset \Hz$ be a sequence,   
s.t. $g^{2,n}_0 \rightarrow g^2$ in $\L$
for $n \rightarrow \infty$ and  let $\{X^{\epsilon,\delta}_{2,n}\}_{n\in\mathbb{N},\delta>0}$  
be a sequence of variational solutions of \eqref{vis.TVF} (for fixed $\eps>0$)
with $x_0\equiv x^{2,n}_0$, $g\equiv g^{2,n}$}. 
We note that the first part of the proof implies that $X^{\epsilon,\delta}_{2,n} \rightarrow X^{\epsilon}_2$
in $L^2(\Omega; C([0,T];\L)$ for $\delta\rightarrow 0$, $n\rightarrow \infty$.
We set $Z=X^{\epsilon,\delta}_{2,n}, G=\Aed(X^{\epsilon,\delta}_{2,n})$ in (\ref{reg.SVI})
and observe that
\begin{align}\label{eps to 0 inequality}
\frac{1}{2} \E{\nos{X^{\epsilon}_1(t)-X^{\epsilon,\delta}_{2,n}(t)}}+&\E{\intt {\bar{\mathcal{J}}_{\eps,\lambda}}(X^{\epsilon}_1(s)) \d s} \nonumber
\\
&\leq  \frac{1}{2} \E{\nos{x^1-x^{2,n}_0}}+\E{\intt {\bar{\mathcal{J}}_{\eps,\lambda}}(X^{\epsilon,\delta}_{2,n}(s)) \d s} \nonumber 
\\
&-{\E{\delta \intt \ska{ X^{\epsilon}_1(s)- X^{\epsilon,\delta}_{2,n}(s), \Delta X^{\epsilon,\delta}_{2,n}(s))}} \d s}
\\
 &-{\E{\intt \ska{ X^{\epsilon}_1 (s)- X^{\epsilon,\delta}_{2,n}(s), \div\fe{X^{\epsilon,\delta}_{2,n}(s)} }\d s}} \nonumber
\\ \nonumber
 &+ \E{\intt \lambda\ska{X^{\epsilon}_1(s)- X^{\epsilon,\delta}_{2,n}(s),X^{\epsilon,\delta}_{2,n}(s)-{g^{2,n}}}  \d s} 
\\ \nonumber
 &+ \E{\intt \nos{X^{\epsilon}_1(s)-X^{\epsilon,\delta}_{2,n}(s)} \d s}
\\\nonumber
& := I+II+III+IV+V+VI.
\end{align}
{The term $III$ is estimated using Young's inequality as
\begin{align*}
III 
&\leq C\E{\intt \delta^{\frac{2}{3}}\nos{ X^{\epsilon}_1(s)-X^{\epsilon,\delta}_{2,n}(s)}+\delta^{\frac{4}{3}}  \nos{\Delta X^{\epsilon,\delta}_{2,n}(s))} \d s}\,.
\end{align*}
{We have to show the following estimate
\begin{align*}
IV=\ska{X_1^{\eps}-X^{\eps,\delta}_{2,n},-\div \fe{X^{\eps,\delta}_{2,n}}} \leq \bar{\mathcal{J}}_{\eps,0}(\Xe_1)-\bar{\mathcal{J}}_{\eps,0}(X^{\eps,\delta}_{2,n}).
\end{align*}
We consider an approximating sequence $x_k \in C^{\infty}(\O)\cap BV(\O)$, s.t., $x_k \rightarrow X_1^{\eps}$ strongly in $L^1(\O)$ and $\Je(x_k) \rightarrow \Je(X_1^{\eps})$ for $k \rightarrow \infty$, cf. \cite[Theorem 13.4.1]{book_attouch}.
Integration by parts then yields
\begin{align*}
&\ska{x_k-X^{\eps,\delta}_{2,n},-\div \fe{X^{\eps,\delta}_{2,n}}}
=\ska{\nabla(x_k-X^{\eps,\delta}_{2,n}), \fe{X^{\eps,\delta}_{2,n}}}
\\
 &+\int_{\partial \O} \gamma_0(x_k) \fe{X^{\eps,\delta}_{2,n}}\nu \d \mathcal{H}^{n-1}-\int_{\partial \O} \gamma_0(X^{\eps,\delta}_{2,n}) \fe{X^{\eps,\delta}_{2,n}}\nu \d \mathcal{H}^{n-1}\,,
\end{align*}
where $\nu $ is the outer normal vector at $\mathcal{H}^{n-1}$  almost all $x \in \partial \O$. Since  $X^{\eps,\delta}_{2,n}(\omega,t)\in\Hz$ for almost all $(\omega,t) \in \Omega \times [0,T]$, the second boundary integral vanishes.
We estimate the first boundary integral as
\begin{align*}
\int_{\partial \O} \gamma_0(x_k) \fe{X^{\eps,\delta}_{2,n}}\nu \d \mathcal{H}^{n-1}
\leq& \int_{\partial \O} \be{\gamma_0(x_k)} \be{\fe{X^{\eps,\delta}_{2,n}}\nu} \ d\mathcal{H}^{n-1}
\\
\leq&\int_{\partial \O} \be{\gamma_0(x_k)}\d \mathcal{H}^{n-1}.
\end{align*}
{On noting that $\no{X_1^{\eps}(\omega,t)} \leq C$  for a.a. $(\omega,t) \in \Omega \times (0,T)$, the convergence  $x_k \rightarrow X_1^{\eps} $ for $k \rightarrow \infty$ in $L^1(\O)$ implies $x_k \weak X_1^{\eps}$ in $\L$ 
a.e. in $\Omega \times (0,T)$.}
We further assert that the trace of each approximating function $x_k \in C^{\infty}(\O)\cap BV(\O)$, coincides with the trace of $X_1^{\eps}$ on the boundary of $\O$, see \cite[Remark 10.2.1]{book_attouch}. 
Hence we obtain by taking the limit for $k\rightarrow \infty$ that
\begin{align*}
\ska{\Xe_1-X^{\eps,\delta}_{2,n},-\div \fe{X^{\eps,\delta}_{2,n}}}&\leq \bar{\mathcal{J}}_{\eps,0}(\Xe_1)-\bar{\mathcal{J}}_{\eps,0}(X^{\eps,\delta}_{2,n}).
\end{align*} 
}
Next, we obtain
\begin{align*}
 V 
 & =  \lambda \E{\intt \ska{X^{\epsilon}_1(s)- X^{\epsilon,\delta}_{2,n}(s),X^{\epsilon,\delta}_{2,n}(s)-{g^{2,n}}} \d s}
\\
 &\leq  \frac{\lambda}{2}\E{\intt \nos{X^{\epsilon}_1(s)-{g^{2,n}}}-\nos{X^{\epsilon,\delta}_{2,n}(s)-{g^{2,n}}}\d s}.
\end{align*}
}
After substituting $III$-$V$ into \eqref{eps to 0 inequality} we arrive at
\begin{align}\label{Uniquness.eps to 0}
\frac{1}{2} \E{\nos{X^{\epsilon}_1(t)-X^{\epsilon,\delta}_{2,n}(t)}}&+\frac{\lambda}{2}\E{\intt \nos{X^{\epsilon}_1(s)-g^1}\d s} \nonumber
\\ \nonumber
&\leq  \frac{1}{2} \E{\nos{x^1_0-x^{2,n}_0}} +\frac{\lambda}{2}\E{\intt \nos{X^{\epsilon,\delta}_{2,n}(s)-g^1}\d s}
\\\nonumber
&+ C\E{\intt \delta^{\frac{2}{3}}\nos{ X^{\epsilon}_1(s)-X^{\epsilon,\delta}_{2,n}(s)}+\delta^{\frac{4}{3}}  \nos{\Delta X^{\epsilon,\delta}_{2,n}(s))} \d s}
\\
& \frac{\lambda}{2}\E{\intt \nos{X^{\epsilon}_1(s)-{g^{2,n}}}\d s} -  \frac{\lambda}{2}\E{\intt \nos{X^{\epsilon,\delta}_{2,n}(s)-{g^{2,n}}}\d s}
\\\nonumber
& +\E{\intt \nos{X^{\epsilon}_1(s)-X^{\epsilon,\delta}_{2,n}(s)} \d s}
.
\end{align}
The convergences (\ref{delta_limit}), (\ref{lim_nnew}) imply the convergence $X^{\epsilon,\delta}_{2,n} \rightarrow X^{\epsilon}_2$ in $L^2(\Omega;C([0,T];\L))$
for  $\delta \rightarrow 0$, $n \rightarrow \infty$. 
We note that for $\delta \rightarrow 0$ the fourth term on the right-hand side of (\ref{Uniquness.eps to 0}) vanishes due to Lemma \ref{laplace_energy_estimate}. 
Hence, by taking the limits for $\delta \rightarrow 0$, $n \rightarrow \infty$ in \eqref{Uniquness.eps to 0}, 
{using the strong convergence $g^{2,n} \rightarrow g^2 $ in $\L$ for $n \rightarrow \infty$} , the lower-semicontinuity of norms
and (\ref{delta_limit}), (\ref{lim_nnew})  we obtain 
\begin{align*}
&\E{\nos{X^{\epsilon}_1(t)-\Xe_2(t)}}\leq C \E{\nos{x^1_0-x^2_0}}
\\
&+\frac{\lambda}{2}\E{\intt\nos{X^{\epsilon}_1(s)-g^2}+\nos{\Xe_2(s)-g^1}-\nos{X^{\epsilon}_{1}(s)-g^1}-\nos{\Xe_2(s)-g^2} \d s}
\\
&+\E{\intt \nos{X^{\epsilon}_1(s)-{X^{\epsilon}_{2}(s)}} \d s}
\\
&\leq C\left(\E{\nos{x^1_0-x^2_0}}+\E{\intt \nos{X^{\epsilon}_1(s)-\Xe_2(s)} \d s} +\nos{g^1-g^2} \right )\, ,
\end{align*}
for all $t \in [0,T]$. 
After applying the Tonelli and Gronwall lemmas we obtain \eqref{reg_stability_inequality}.
\end{proof}
Our second main theorem establishes existence and uniqueness of a SVI solution to \eqref{TVF}
in the sense of Definition~\ref{def_varsoleps}.
The solution is obtained as a limit of solutions of the regularized gradient flow \eqref{reg.TVF}
for $\epsilon \rightarrow 0$.
\begin{thms}\label{Thm.SVI}
Let $0  < T < \infty$ and {$x_0 \in L^2(\Omega,\F_0;\L)$,   $ g \in \L$} be fixed. 
Let $\{\Xe\}_{\epsilon>0}$ be the {SVI} solutions of \eqref{reg.TVF} for $\epsilon\in(0,1]$.
Then  $\Xe$ converges to the unique SVI variational solution $X$ of (\ref{TVF})
in $L^2(\Omega;C([0,T];\L))$ for $\epsilon\rightarrow 0$, i.e., there holds
\begin{align}\label{epsilon goes to 0}
\lim\limits_{\epsilon \rightarrow 0} \E{\sup\limits_{t \in [0,T]}\nos{\Xe(t)-X(t)}}=0.
\end{align}
Furthermore, the following estimate holds
\begin{align}\label{stability_inequality}
&\E{\nos{X_1(t)-X_2(t)}}\leq C\left(\E{\nos{x^1_0-x^2_0}}+\nos{g^1-g^2} \right)\quad \mathrm{for\,\,all\,\,} t \in [0,T]\,,
\end{align}
where $X_1$ and $X_2$ are SVI solutions of (\ref{TVF}) with $x_0\equiv x^1_0$, $g\equiv g^1$  and $x_0\equiv x^2_0$, $g\equiv g^2$, respectively.
\end{thms}
 \begin{proof}[\textbf{Proof of Theorem \ref{Thm.SVI}}]\hfill \\
We consider $\L$-approximating sequences {$\{x_0^n\}_{n\in\mathbb{N}} \subset L^2(\Omega,\F_0;\Hz)$ 
and $\{g^n\}_{n\in\mathbb{N}} \subset \Hz$ of the initial condition $x_0\in L^2(\Omega,\F_0;\L)$ and $g \in \L$, respectively.}
For $n\in\mathbb{N}$, $\delta>0$ we denote by $\Xdno,\Xdnt$ the variational solutions of \eqref{vis.TVF}
with $\epsilon\equiv\epsilon_1$, $\epsilon\equiv\epsilon_2$, respectively. 
By It\^o's formula the difference satisfies
\begin{align}\label{eps.difference}
\frac{1}{2}& \nos{\Xdno(t)-\Xdnt(t)}
 \nonumber\\
=& -\delta \intt \nos{\nabla (\Xdno(s)-\Xdnt(s))}\d s
 \nonumber \\
&-\intt \ska{\frac{\nabla \Xdno(s)}{ \sqrt{ \vert \nabla \Xdno(s) \vert^2 +\epsilon_1^2}}-\frac{\nabla \Xdnt(s)}{ \sqrt{ \vert \nabla \Xdnt(s) \vert^2 +\epsilon_2^2}},\nabla (\Xdno(s)-\Xdnt(s))}\d s
  \\
&-\lambda\intt \nos{(\Xdno(s)-\Xdnt(s)}\d s +\intt \nos{\Xdno(s)-\Xdnt(s)}\d W(s)\nonumber
\\
 &+\intt \nos{\Xdno(s)-\Xdnt(s)} \nonumber\d s.  
\end{align}
We estimate the second term on the right-hand side of \eqref{eps.difference} using the convexity \eqref{eps.convexity.inequality} 
\begin{align}\label{est_convex}
 & \ska{\frac{\nabla \Xdno(s)}{ \sqrt{ \vert \nabla \Xdno(s) \vert^2 +\epsilon_1^2}}-\frac{\nabla \Xdnt(s)}{ \sqrt{ \vert \nabla \Xdnt(s) \vert^2 +\epsilon_2^2}},\nabla (\Xdno(s)-\Xdnt(s))}
\nonumber \\
 =&\ska{\frac{\nabla \Xdno(s)}{ \sqrt{ \vert \nabla \Xdno(s) \vert^2 +\epsilon_1^2}},\nabla (\Xdno(s)-\Xdnt(s))}
\nonumber \\
 &+\ska{\frac{\nabla \Xdnt(s)}{ \sqrt{ \vert \nabla \Xdnt(s) \vert^2 +\epsilon_2^2}},\nabla (\Xdnt(s)-\Xdno(s))}
\\
 \geq& \into \sqrt{\bes{\nabla \Xdno}+\epsilon_1^2}-\sqrt{\bes{\nabla \Xdnt}+\epsilon_1^2} \d \xx 
\nonumber\\
\nonumber
 &+\into \sqrt{\bes{\nabla \Xdnt}+\epsilon_2^2}-\sqrt{\bes{\nabla \Xdno}+\epsilon_2^2}   \d \xx.
 \end{align}
 Next, we observe that
 \begin{align*}
\into& \left(\sqrt{\bes{\nabla \Xdno}+\epsilon_1^2}-\sqrt{\bes{\nabla \Xdno}+\epsilon_2^2}\right)  \d \xx
\\
 &=\into \frac{\left(\sqrt{\bes{\nabla \Xdno}+\epsilon_1^2}-\sqrt{\bes{\nabla \Xdno}+\epsilon_2^2}\right)\left(\sqrt{\bes{\nabla \Xdno}+\epsilon_1^2}+\sqrt{\bes{\nabla \Xdno}+\epsilon_2^2}\right)}{\sqrt{\bes{\nabla \Xdno}+\epsilon_1^2}+\sqrt{\bes{\nabla \Xdno}+\epsilon_2^2}}\d \xx
\\
 &= \into  \frac{\bes{\nabla \Xdno}+\epsilon_1^2-\bes{\nabla \Xdno}-\epsilon_2^2}{\sqrt{\bes{\nabla \Xdno}+\epsilon_1^2}+\sqrt{\bes{\nabla \Xdno}+\epsilon_2^2}}\d \xx
\\
 &= \into \frac{(\epsilon_1+\epsilon_2)(\epsilon_1-\epsilon_2)}{\sqrt{\bes{\nabla \Xdno}+\epsilon_1^2}+\sqrt{\bes{\nabla \Xdno}+\epsilon_2^2}}\d \xx
\\
 &\leq \into \be{\epsilon_1 -\epsilon_2}\left( \frac{\epsilon_1}{\sqrt{\bes{\nabla \Xdno}+\epsilon_1^2}}+ \frac{\epsilon_2}{\sqrt{\bes{\nabla \Xdno}+\epsilon_2^2}} \right)\d \xx 
 \leq C(\epsilon_1+\epsilon_2).
 \end{align*}
 Using the inequality above, we get 
 \begin{align*}
 &\into \sqrt{\bes{\nabla \Xdno}+\epsilon_1^2}-\sqrt{\bes{\nabla \Xdnt}+\epsilon_2^2} \d \xx 
 +\into \sqrt{\bes{\nabla \Xdnt}+\epsilon_2^2}-\sqrt{\bes{\nabla \Xdno}+\epsilon_2^2}   \d \xx\\
 &\geq - \left|\into \sqrt{\bes{\nabla \Xdno}+\epsilon_1^2}-\sqrt{\bes{\nabla \Xdno}+\epsilon_2^2} \d \xx \right|\\
 &-\left|\into \sqrt{\bes{\nabla \Xdnt}+\epsilon_1^2}-\sqrt{\bes{\nabla \Xdnt}+\epsilon_2^2}   \d \xx\right| \\
 &\geq -C(\epsilon_1+\epsilon_2).
 \end{align*}
Substituting (\ref{est_convex}) along with the last inequality into \eqref{eps.difference} yields
\begin{align}\label{eps.estimate}
\frac{1}{2} \nos{\Xdno(t)-\Xdnt(t)} \leq& C(\epsilon_1+\epsilon_2) \nonumber\\
&+\intt \nos{\Xdno(s)-\Xdnt(s)}\d W(s)\\
 &+\intt \nos{\Xdno(s)-\Xdnt(s)}\nonumber\d s.  
\end{align}
After using the Burkholder-Davis-Gundy inequality for $p=1$, the Tonelli and Gronwall lemmas we obtain that
\begin{align}\label{Cauchy eps}
\E{\sup\limits_{t \in [0,T]}\nos{\Xdno(t)-\Xdnt(t)}}\leq C(\epsilon_1+\epsilon_2).
\end{align}
We take the limit for $\delta\rightarrow 0$ in \eqref{Cauchy eps}
for fixed $n$ and $\epsilon_1,\epsilon_2$, 
and obtain using \eqref{delta_limit} by the lower-semicontinuity of norms that
\begin{align}\label{delta to 0}
\E{\sup\limits_{t \in [0,T]} \nos{X^{\epsilon_1}_n(t)-X^{\epsilon_2}_n(t)}}
 \leq& \liminf \limits_{\delta \rightarrow 0 }
\E{\sup\limits_{t \in [0,T]} \nos{X^{\epsilon_1,\delta}_n(t)-X^{\epsilon_2,\delta}_n(t)}} \nonumber \\
\leq& C(\epsilon_1+\epsilon_2).
\end{align}
Hence, by \eqref{lim_nnew} and the lower-semicontinuity of norms, after taking the limit $n\rightarrow \infty$ in \eqref{delta to 0}
for fixed $\epsilon_1$, $\epsilon_2$ we get
\begin{align}\label{n to infty}
\E{\sup\limits_{t \in [0,T]} \nos{X^{\epsilon_1}(t)-X^{\epsilon_2}(t)}}
 \leq& \liminf \limits_{n \rightarrow \infty }
\E{\sup\limits_{t \in [0,T]} \nos{X^{\epsilon_1}_n(t)-X^{\epsilon_2}_n(t)}}
\\ \nonumber
\leq& C(\epsilon_1+\epsilon_2).
\end{align}
The above inequality implies that $\{X^{\epsilon}\}_{\epsilon>0}$ is a Cauchy Sequence in $\epsilon$.
 Consequently there exists a unique $\{\F_t\}$-adapted process $X \in   L^2(\Omega;C([0,T];\L))$ with $X(0)=x_0 $  such that
 \begin{align}\label{epsilon gegen 0}
\lim\limits_{\epsilon \rightarrow 0}\E{\sup\limits_{t \in [0,T]} \nos{X^{\epsilon}(t)-X(t)}}
 = 0.
\end{align}
This concludes the proof of \eqref{epsilon goes to 0}.

Next, we show that the limiting process $X$ is the SVI solution of (\ref{TVF}), i.e., we show that (\ref{SVIeps0}) holds.
We note that \eqref{BV-eps-esti.} implies that
\begin{align}\label{BVestimate}
\sup\limits_{\epsilon \in (0,1]} \E{\intt \bar{\mathcal{J}}_{\lambda}(\Xe(s)\d s } \leq C. 
\end{align}
Hence using \eqref{epsilon gegen 0}, \eqref{BVestimate} we get by 
Fatou's lemma and \cite[Proposition 11.3.2]{book_attouch} that
\begin{align*}
\liminf\limits_{\epsilon \rightarrow 0} \E{ \intt \bar{\mathcal{J}}_{\lambda}(\Xe(s)) \d s } \geq  \E { \intt \bar{\mathcal{J}}_{\lambda}(X(s)) \d s} .
\end{align*}
By Theorem~\ref{Thm_reg.SVI} we know that $X^{\epsilon}$ satisfies (\ref{reg.SVI}) for any $\epsilon \in (0,1]$.
By taking the limit for $\epsilon\rightarrow 0$ in (\ref{reg.SVI}),
using the above inequality and (\ref{epsilon gegen 0}) 
it follows that $X$ satisfies $\eqref{SVIeps0}$.
Finally, inequality $\eqref{stability_inequality}$ follows after taking the limit for $\epsilon \rightarrow 0$ in
(\ref{epsilon gegen 0}), by \eqref{reg_stability_inequality} and the lower semicontinuity of norms.
\end{proof}


\section{Numerical Approximation}\label{sec_num}
We construct a fully-discrete approximation of the STVF equation \eqref{TVF}
via an implicit time-discretization of the regularized STVF equation \eqref{reg.TVF}. 
For $N \in \N$ we consider the time-step $\tau := T/N$,
set $t_i:=i\tau$  for $i=0,\ldots,N$ and denote the discrete Wiener increments as $\Delta_i W:= W(t_i)-W(t_{i-1})$.
We combine the discretization in time with a the standard $\mathbb{H}^1_0$-conforming finite element method,
see, e.g., \cite{BrennerS02}, \cite{Prohl_TVF_numerics}, \cite{bm16}.
Given a family of quasi-uniform triangulations $\big\{\mathcal{T}_h\big\}_{h>0}$ of $\D$ into open simplices with mesh size $h=\max_{K\in \mathcal{T}_h}\{\mathrm{diam}(K)\}$
we consider the associated space of piecewise linear, globally continuous functions 
$\mathbb{V}_h = \{v_h \in C^0(\overline{\D});\, v_h|_K \in \mathcal{P}^1(K)\,\, \forall K\in \mathcal{T}_h\}\subset \mathbb{H}^1_0$
and set $L\equiv\text{dim}\mathbb{V}_h$ for the rest of the paper.
We set $X^h_0:=\mathcal{P}_h x_0$, $g^h:=\mathcal{P}_h g$, where $\mathcal{P}_h$ is the $\L$-projection onto $\mathbb{V}_h$.
%

The implicit fully-discrete approximation of (\ref{reg.TVF}) is defined as follows:
fix $N\in\N$, $h>0$ set $X^0_\varepsilon = x^h_0\in\mathbb{V}_h$ and determine $\Xi\in \mathbb{V}_h$, $i=1,\dots, N$ as the solution of
\begin{align}\label{num.reg.TVF}
\ska{\Xi,\vh}&=\ska{\Xmin,\vh}-\tau \ska{\fe{\Xi},\nabla\vh } \\
&-\tau\lambda\ska{\Xi -g^h,\vh}+\ska{\Xmin,\vh}\Delta_i W &&\forall \vh \in \mathbb{V}_h \nonumber.
\end{align}

To show convergence of the solution of the numerical scheme \eqref{num.reg.TVF}
we need to consider a discretization of the regularized problem (\ref{vis.TVF}).
Given $x_0\in L^2(\Omega,\F_0;\L)$, $g \in \L$
and $n \in \N$ we choose $x^n_0:=\mathcal{P}_n x_0\in\mathbb{V}_n$, $g^n:=\mathcal{P}_n g\in\mathbb{V}_n$ in (\ref{vis.TVF}).
Since $\mathbb{V}_n\subset \Hz$ the sequences $\{x^n_0\}_{n\in \mathbb{N}} \subset L^2(\Omega,\F_0;\Hz)$, $\{g_n\}_{n\in\N} \in \Hz$
constitute $\Hz$-approximating sequences of $x_0\in L^2(\Omega,\F_0;\L)$, $g \in \L$, respectively.
We set $x^{h,n}_0:=\mathcal{P}_h x_0^n$, $g^{h,n}:=\mathcal{P}_h g^n$, where $\mathcal{P}_h$ is the $\L$-projection onto $\mathbb{V}_h$. 
The fully-discrete Galerkin approximation of (\ref{vis.TVF}) for fixed $n \in \N $  is then defined as follows:  
fix $ N\in \N$, $h>0$ set $X_{\epsilon,\delta,n,h}^{0}=x^{h,n}_0$
and determine $\Yi\in \mathbb{V}_h$, $i=1,\dots, N$ as the solution of
\begin{align}\label{num.visc.TVF}
\ska{\Yi,\vh}&= \ska{\Ymin,\vh}-\tau \delta \ska{\nabla \Yi,\nabla \vh}-\tau \ska{\fe{\Yi},\nabla\vh }
\nonumber \\
&-\tau\lambda\ska{\Yi -g^{h,n},\vh}+\ska{\Ymin,\vh}\Delta_i W &&\forall \vh \in \mathbb{V}_h .
\end{align}


{
The next lemma, cf. \cite[Lemma II.1.4]{book_temam} is used to show $\mathbb{P}$-a.s. existence of discrete solutions $\{\Xi\}_{i=1}^N$, $\{\Yi\}_{i=1}^N$ of numerical schemes \eqref{num.reg.TVF}, \eqref{num.visc.TVF}, respectively. 
 \begin{lems}\label{lemma_Existence_for_numerical_scheme}
  Let $h: \R^L \rightarrow \R^L$ be continuous. If there is $R>0 $ such that $h(v)v\geq 0$ whenever $\no{v}_{\R^L}=R $ then there exist $\bar{v}$ satisfying $\no{\bar{v}}_{\R^L} \leq R$ and $h(\bar{v})=0$.
 \end{lems}
 In order to show $\{\mathcal{F}_{t_i}\}_{i=1}^{N}$-measurability of the random variables $\{\Xi\}_{i=1}^N$, $\{\Yi\}_{i=1}^N$
we make use of the following lemma, cf. \cite{gm_05,em_sis_18}.
\begin{lems}\label{lemma_measurability_for_numerical_scheme}
 Let $(S,\Sigma)$ be a measure space. Let $f :S\times \R^L\rightarrow \R^L$ be a function that is $\Sigma$-measurable in its first argument for every $x \in \R^L$, that is continuous in its second argument for every $\alpha \in S$ and moreover such that for 
 every $\alpha \in S$ the equation $f(\alpha, x)=0$ has an unique solution $x=g(\alpha)$. Then $g : S \rightarrow \R^L$ is $\Sigma$-measurable.
\end{lems}
}

Below we show the existence, uniqueness and measurability of numerical solutions of (\ref{num.reg.TVF}), (\ref{num.visc.TVF}).
We state the result for the scheme (\ref{num.visc.TVF}) only,
since the proof also holds for $\delta=0$ (i.e. for (\ref{num.reg.TVF})) without any modifications.
\begin{lems}\label{Lemma_Existence_measurability_num.Scheme}
 Let $x_0 \in L^2(\Omega,\F_0;\L)$, $g\in \L$
and let $L,n,N \in \N$ be fixed.
The for any $\delta\geq 0$, $\epsilon > 0$, $i=1,\ldots,N,$ there exist $\mathcal{F}_{t_i}$-measurable $\mathbb{P}$-a.s. unique 
random variables $\Yi\in\mathbb{V}_h$ which solves (\ref{num.visc.TVF}).
\end{lems}
\begin{proof}[\textbf{Proof of Lemma \ref{Lemma_Existence_measurability_num.Scheme}}]
Assume that the $\mathbb{V}_h$-valued random variables $X_{\eps,\delta,n,h}^{0},\ldots,X_{\eps,\delta,n,h}^{i-1}$ satisfy \eqref{num.visc.TVF}
 and that $X_{\eps,\delta,n,h}^{k}$ is $\F_{t_k}$-measurable for $k=1,\ldots,i-1$. 
We show that there is a  $\F_{t_i}$ measurable random variable $\Yi$, that satisfies \eqref{num.reg.TVF}. 
Let $\{\phi_\ell\}_{\ell=1}^L$ be the basis of $\mathbb{V}_h$.
We identify every $v \in \mathbb{V}_h$ with a vector $\bar{v} \in \R^L$ with $v = \sum_{\ell=1}^L \bar{v}_\ell \phi _\ell$ 
and define a norm on $\R^L$ as $\no{\bar{v}}_{\R^L}:= \no{v}_{\Hz}$. 
For an arbitrary $\omega \in \Omega$ we represent $X_\omega \in \mathbb{V}_h$ as a vector $\bar{X}_\omega \in \R^L$
and define a function $h: \Omega \times \R^L \rightarrow \R^L$ component-wise for $\ell=1,\ldots,L$ as
\begin{align*}
  h(\omega,\bar{X}_\omega)_\ell:=(X_\omega-\Ymin(\omega),\phi_\ell)+\tau\delta	(\nabla X_\omega, \nabla \phi_\ell)+\tau(\fe{X_\omega},\nabla \phi_\ell)\\
+\tau\lambda(X_\omega- g^{h,n},\phi_\ell) -(\Ymin(\omega),\phi_\ell)\Delta_i W(\omega).
 \end{align*}
We show, that for each $\omega \in \Omega$ there exists an $\bar{X}_\omega$ such that $h(\omega,\bar{X}_\omega)=0$. 
We note the following inequality
\begin{align*}
  h(\omega,\bar{X}_\omega)\cdot \bar{X}_\omega=&(X_\omega-\Ymin(\omega),X)+\tau\delta\nos{\Delta X_\omega}+\tau(\fe{X_\omega},\nabla X_\omega)
\\
&+\tau\lambda(X_\omega- g^{h,n},X_\omega) -(\Ymin(\omega),X_\omega)\Delta_i W(\omega)
\\
 \geq & \nos{X_\omega} -(\Ymin(\omega),X_\omega)+\tau(\fe{X_\omega},\nabla X_\omega)\\\
  &-(\Ymin(\omega),X_\omega)\Delta_i W+\tau\lambda \no{X_\omega}  -\tau\lambda ( g^n,X_\omega)\\
  \geq & \no{ X_\omega}\left(  \no{ X_\omega} -\no{ \Ymin(\omega)} -\no{\Ymin(\omega)}|\Delta_i W(\omega)|- \no{ g^{h,n}}     \right)\,.
 \end{align*}
On choosing $\no{X_\omega}=R_\omega$ large enough, 
the existence of $X_{\eps,\delta,n}^{i}(\omega)\in \mathbb{V}_h$ for each $\omega\in\Omega$
then follows by Lemma \ref{lemma_Existence_for_numerical_scheme}, since $h(\omega,\cdot)$ is continuous 
by the demicontinuity of the operator $\Aed$, which follows from  hemicontinuity and and monotonicty of $\Aed$ for $\delta\geq 0$, $\epsilon >0$, see \cite[Remark 4.1.1]{Roeckner_book}.
The $\mathcal{F}_{t_i}$-measurabilty follows by Lemma~\ref{lemma_measurability_for_numerical_scheme} for unique $X_{\eps,\delta,n}^{i}$. 

Hence, it remains to show that $X_{\eps,\delta,n,h}^{i}$ is $\mathbb{P}$-a.s. unique. 
Assume there are two different solution $X_{1},X_{2}$, s.t. $h(\omega,\overline{X}_{1}(\omega))=0=h(\omega,\overline{X}_{2}(\omega))$ for $\omega\in\Omega$.
Then by the convexity (\ref{eps.convexity.inequality}) we observe that
 \begin{align*}
  0=&\big(h(\omega,\overline{X}_1(\omega))-h(\omega,\overline{X}_2(\omega))\big)\cdot\big(\overline{X}_1(\omega)-\overline{X}_2(\omega)\big)\\
  =&(1+\tau \lambda)\nos{ X_1(\omega) - X_2(\omega)}+\tau \delta \nos{\nabla (X_1-X_2)(\omega)}\\
  &+\tau\left(\fe{X_1(\omega)} -\fe{X_2(\omega)},\nabla X_1(\omega) - \nabla X_2(\omega)\right)\\
  \geq& (1+\tau \lambda )\nos{X_1(\omega)-X_2(\omega)}+\tau \delta \nos{\nabla(X_1-X_2)(\omega)}\,.
 \end{align*}
Hence ${X}_1\equiv {X}_2$ $\mathbb{P}$-a.s.
\end{proof}
We define the discrete Laplacian $\Delta_h : \mathbb{V}_h \rightarrow \mathbb{V}_h$ by
\begin{align}\label{disc_Lap}
-\ska{\Delta_h w_h,v_h}=\ska{\nabla w_h,\nabla v_h} ~~\forall w_h,v_h \in \mathbb{V}_h. 
\end{align}
To obtain the required the stability properties of the numerical approximation \eqref{num.visc.TVF} we need the following lemma.
\begin{lems}\label{Lemma_Discret_Laplace}
Let $\Delta_h$ be the discrete Laplacian defined by \eqref{disc_Lap}. Then for any $v_h \in \mathbb{V}_h$, $\varepsilon,h>0$ the following inequality holds:
\begin{align}
-\ska{\fe{v_h},\nabla (\Delta_h \vh)} \geq 0\,.
\end{align}
\end{lems} 
\begin{proof}[\textbf{Proof of Lemma \ref{Lemma_Discret_Laplace}}]
Let $\{\phi_\ell\}_{\ell=1}^L$ be the basis of $\mathbb{V}_h$ consisting of continuous piecewise linear Lagrange basis functions
associated with the nodes of the triangulation $\mathcal{T}_h$. Then any $\vh \in \mathbb{V}_h$ has the 
representation $\vh=\sum_{\ell=1}^L (\vh)_{\ell} \phi_{\ell}$, where $(\vh)_{\ell} \in \R, \ell=1,\ldots,L $
and analogically  $\Delta_h\vh=\sum_{\ell=1}^L (\Delta_h \vh)_\ell \phi_\ell$, with coefficients $(\Delta \vh)_\ell \in \R, \ell=1,\ldots,L$. 
From (\ref{disc_Lap}) it follows that
\begin{align}\label{coeff}
\sum_{\ell=1}^L (\Delta \vh)_\ell \ska{\phi_\ell,\phi_k}
=-\sum_{\ell=1}^L (\vh)_{\ell} \ska{\nabla \phi_{\ell},\nabla \phi_k}=-\sum_{K\in \mathcal{T}_h}\sum_{\ell=1}^L (\vh)_{\ell} \ska{\nabla \phi_{\ell},\nabla \phi_k}_{K}\,,
\end{align}
where we denote $(v,w)_K := \int_{K}v(x)w(x)\mathrm{d}x$ for $K\in \mathcal{T}_h$.

We rewrite \eqref{coeff} with the mass matrix $M:=\{M\}_{i,k}:=(\phi_i,\phi_k)$ and the stiffness matrix $A:=\{A\}_{i,k}:=(\nabla \phi_i,\nabla \phi_k)$  and 
$A_K := \{A_K\}_{i,k}:=(\nabla \phi_i,\nabla \phi_k)_{K}$ as
\begin{align}\label{coeff_Matrix}
\Delta_h \bar{v}_h=M^{-1}A  \bar{v}_h=M^{-1}\sum_{K \in \mathcal{T}_h} A_K  \bar{v}_h\, ,
\end{align} 
where $\Delta_h\bar{v}_h\in \R^L$ is the vector $((\Delta_h v_h)_1,\ldots,(\Delta_h v_h)_L)^T$ and 
$\bar{v}_h\in \R^L$ is the vector $((v_h)_1,\ldots,(v_h)_L)^T$. 
Since $\mathbb{V}_h$ consists of functions, which are piecewise linear on the triangles $K \in \mathcal{T}_h$, $(\bes{\nabla \vh}+\epsilon^2)^{-\frac{1}{2}}$ is constant on every triangle $T$. 
We note, that the matrices $M$ and $M^{-1}$ are positive definite.
We get using the Young's inequality 
\begin{align*}
-&\ska{\fe{\vh},\nabla (\Delta_h \vh)}
=-\sum_{K\in \mathcal{T}_h}(\bes{\nabla \vh}+\eps^2)^{-\frac{1}{2}}_K\ska{\nabla \vh,\nabla (\Delta_h \nabla\vh)}_K
\\
=&-\sum_{K\in \mathcal{T}_h}\sum_{k,\ell=1}^L(\bes{\nabla \vh}+\eps^2)^{-\frac{1}{2}}_K(v_h)_{\ell}(\Delta_h v_h)_k\ska{\nabla\phi_{\ell},  \nabla \phi_k }_K
\\
=&-\sum_{K\in \mathcal{T}_h}(\bes{\nabla \vh}+\eps^2)^{-\frac{1}{2}}_K \bar{v}_h^T A^T_K\Delta_h\bar{v}_h
\\
=&\sum_{K,K' \in \mathcal{T}_h}(\bes{\nabla \vh}+\eps^2)^{-\frac{1}{2}}_K \bar{v}_h^TA^T_K M^{-1} A_{K'}\bar{v}_h
\\
=&\frac{1}{2}\sum_{K,K'\in \mathcal{T}_h}(\bes{\nabla \vh}+\eps^2)^{-\frac{1}{2}}_K \bar{v}_h^TA^T_K M^{-1}A_{K'}\bar{v}_h\\
&+\frac{1}{2}\sum_{K,K'\in \mathcal{T}_h}(\bes{\nabla \vh}+\eps^2)^{-\frac{1}{2}}_{K'} \bar{v}_h^TA^T_K M^{-1}A_{K'}\bar{v}_h
\\
=&\frac{1}{2}\sum_{K,K'\in \mathcal{T}_h} \bar{v}_h^T A^T_K M^{-1}A_{K'}\bar{v}_h\left((\bes{\nabla \vh}+\eps^2)^{-\frac{1}{2}}_K+(\bes{\nabla \vh}+\eps^2)^{-\frac{1}{2}}_{K'} \right)\\
\\
\geq&\frac{1}{2}\sum_{K,K'\in \mathcal{T}_h} \sqrt{(\bes{\nabla \vh}+\eps^2)^{-\frac{1}{2}}_K}\bar{v}_h^T A^T_K M^{-1}A_{K'}\bar{v}_h\sqrt{(\bes{\nabla \vh}+\eps^2)^{-\frac{1}{2}}_{K'}} \\
\geq& 0\, ,
\end{align*}
since $M^{-1}$ is positive definite.
\end{proof}

In the next lemma we state the stability properties of the numerical solution of the scheme (\ref{num.visc.TVF})
which are discrete analogues of estimates in Lemma~\ref{lem_visc_exist}~and Lemma~\ref{laplace_energy_estimate}.
\begin{lems}\label{Lemma_Discrete a priori estimates}
 Let $x_0 \in L^2(\Omega,\F_0;\L)$ and $g\in\L$ be given.
Then there exists a constant $C \equiv C(\E{\|x_0\|_{\L}}, \|g\|_{\L}) > 0$ such that for any $n \in \N$, $\tau,h>0$
the solution of scheme (\ref{num.visc.TVF}) satisfies
\begin{align}\label{discrete_energy_estimate_viscTVF}
\sup_{i=1,\ldots,N}\E{\nos{\Yi}}&+\frac{1}{4}\E{\sum_{k=1}^N\nos{X_{\eps,\delta,n,h}^{k}-X_{\eps,\delta,n,h}^{k-1}}} \nonumber \\+& \tau \delta \E{\sum_{k=1}^N\nos{\nabla X_{\eps,\delta,n,h}^{k}}}+\frac{\tau \lambda}{2}\E{\sum_{k=1}^N \nos{X_{\eps,\delta,n,h}^{k}}}\leq  C\,,
\end{align}
and a constant $C_{n} \equiv C( \mathbb{E}[\|x_0^n\|_{\Hz}], \|g^n\|_{\Hz}) > 0$ such that for any $\tau,h>0$
\begin{align}\label{discrete_H1_estimate_viscTVF}
\sup_{i=1,\ldots,N}\E{\nos{\nabla \Yi}}+\frac{1}{4}\E{\sum_{k=1}^N\nos{\nabla(X_{\eps,\delta,n,h}^{k}-X_{\eps,\delta,n,h}^{k-1})}}+\tau \delta \E{\sum_{k=1}^N\nos{\Delta_h  X_{\eps,\delta,n,h}^{k}}}\leq  C_{n}.
\end{align}
\end{lems}
\begin{proof}[\textbf{Proof of Lemma \ref{Lemma_Discrete a priori estimates}}]
We set $\vh=\Yi$  in \eqref{num.visc.TVF}, use the identity
$2(a-b)a =a^2 - b^2 + (a-b)^2$
and get for $i=1,\ldots,N$
\begin{align}\label{num.visc.energie.estimate}
\frac{1}{2}\nos{\Yi} +\frac{1}{2}\nos{\Yi-\Ymin}+\tau \delta \nos{\nabla \Yi}+\tau \ska{\fe{\Yi},\nabla \Yi} \nonumber\\
=\frac{1}{2}\nos{\Ymin}-\tau \lambda\left(\nos{\Yi}-\ska{ g^{h,n},\Yi}\right) +\ska{\Ymin,\Yi}\Delta_i W. 
\end{align}
We take expected value in (\ref{num.visc.energie.estimate})
and on noting the properties of
Wiener increments $\E{\Delta_i W}=0$, $\E{ \bes{\Delta_i W}}=\tau$ and the independence of $\Delta_i W$ and $\Ymin$
we estimate the stochastic term as
\begin{align*}
\E{\ska{\Ymin,\Yi}\Delta_i W}&=\E{\ska{\Ymin,\Yi-\Ymin}\Delta_i W}+\E{\ska{\Ymin,\Ymin}\Delta_i W}
\\
\leq& \E{\frac{1}{4}\nos{\Ymin-\Yi}+\nos{\Ymin}\bes{\Delta_i W}} + \E{\|\Ymin\|^2}\E{\Delta_i W}
\\
=&\frac{1}{4}\E{\nos{\Yi-\Ymin}}+\tau\E{\nos{\Ymin}}.
\end{align*}
We neglect the positive term 
\begin{align*}
\ska{\fe{\Yi},\nabla \Yi} \geq 0\, ,
\end{align*}
and get from (\ref{num.visc.energie.estimate}) that
\begin{align*}
\frac{1}{2}\E{\nos{\Yi}} +&\frac{1}{4}\E{\nos{\Yi-\Ymin}}+\tau \delta \E{\nos{\nabla \Yi}}+\frac{\tau\lambda}{2}\E{\nos{\Yi}}\\
\leq& \frac{1}{2}\E{\nos{\Ymin}}+\tau\E{\nos{\Ymin}}+\tau\lambda{\nos{g^{h,n}}}\,. \nonumber
\end{align*}
We sum up the above inequality for $k=1,\ldots,i$ and obtain
\begin{align}\label{num.visc.energie.estimate3}
\frac{1}{2}\E{\nos{\Yi}} +&\frac{1}{4}\E{\sum_{k=1}^i\nos{X_{\eps,\delta,n,h}^{k}-X_{\eps,\delta,n,h}^{k-1}}}+\tau \delta \E{\sum_{k=1}^i\nos{\nabla X_{\eps,\delta,n,h}^{k}}}+\frac{\tau \lambda}{2}\E{\sum_{k=1}^i \nos{ X_{\eps,\delta,n,h}^{k}}}
\nonumber \\ 
\leq& \frac{1}{2}\E{\nos{x^n_0}}+\tau\E{\sum_{k=1}^{i}\nos{X_{\eps,\delta,n,h}^{k-1}}} +T\lambda{\nos{  g^{h,n}}}\,.
\end{align}
By the discrete Gronwall lemma it follows from (\ref{num.visc.energie.estimate3}) that
\begin{align*}
\sup\limits_{i=1,\ldots,N}\E{\nos{ \Yi}} \leq \exp(2T)\Big(\E{\nos{x_0}}+2T\lambda{\nos{g}}\Big).
\end{align*}
We substitute the above estimate into the right-hand side of \eqref{num.visc.energie.estimate3} 
to conclude (\ref{discrete_energy_estimate_viscTVF}).
To show the estimate
(\ref{discrete_H1_estimate_viscTVF}) we set $v_h = \Delta_h \Yi$ in (\ref{num.visc.TVF}) use integration by parts
and proceed analogically to the first part of the proof.
We note that by Lemma~\ref{Lemma_Discret_Laplace} it holds that
\begin{align}\label{discrete_resolvent_estimate}
 \ska{\fe{\Yi},\nabla \Delta_h \Yi} \geq 0.
\end{align}
Hence we may neglect the positive term and get that
\begin{align*}
\frac{1}{2}&\E{\nos{\nabla \Yi}} +\frac{1}{4}\E{\sum_{k=1}^i\nos{\nabla(X_{\eps,\delta,n,h}^{k}-X_{\eps,\delta,n,h}^{k-1})}}+\tau \delta \E{\sum_{k=1}^i\nos{\Delta_h X_{\eps,\delta,n,h}^{k}}}\nonumber\\
&+\frac{\tau \lambda}{2}\E{\sum_{k=1}^i \nos{ \nabla X_{\eps,\delta,n,h}^{k}}}
\leq  \frac{1}{2}\E{\nos{\nabla x^n_0}}+\tau\E{\sum_{k=1}^{i}\nos{\nabla X_{\eps,\delta,n,h}^{k-1}}} +T\lambda{\nos{\nabla   g^n}}\,.
\end{align*}
and obtain (\ref{discrete_H1_estimate_viscTVF}) after an application of the discrete Gronwall lemma.
\end{proof}
We define piecewise constant time-interpolants of the numerical solution
$\{\Yi\}_{i=0}^{N}$ of (\ref{num.visc.TVF}) for $t\in [0,T]$ as
\begin{align}\label{eps_delta_interpol1}
\Yc(t):= \Yi \quad  \mathrm{if}\quad t\in (t_{i-1},t_i]
\end{align}
and
\begin{align}\label{eps_delta_interpol2}
\Ycm(t):= \Ymin\quad \mathrm{if}\quad t\in [t_{i-1},t_i)\,.
\end{align}
We note that (\ref{num.visc.TVF}) 
can be reformulated as
\begin{align}\label{Integralformulation}
 &\ska{\Yc(t),\vh}+\left\langle\int_0^{\theta_{+}(t)} \Aed\Yc(s) \d s,\vh\right\rangle \nonumber \\
 &=\ska{X_{\eps,\delta,n}^0,\vh}+\ska{\int_0^{\theta_+(t)} \Ycm(s) \d W(s),\vh} \qquad \mathrm{for}\,\, t\in [0,T],
\end{align}
where $\theta_+(0):=0$ and $\theta_+(t):=t_i$ if  $t\in (t_{i-1},t_{i}]$.

Estimate \eqref{discrete_energy_estimate_viscTVF} 
yields the bounds
\begin{align}\label{Constant interpolation estimate}
\sup\limits_{t\in [0,T]}\E{\nos{\Yc(t)}} &\leq C, &\sup\limits_{t\in [0,T]}\E{\nos{\Ycm(t)}} \leq C,\\
~ \delta \E{\int_0^T \nos{\nabla\Yc(s)}\d s} &\leq C .\nonumber
\end{align}
Furthermore, \eqref{Constant interpolation estimate} and \eqref{a_bnd} imply
\begin{align}\label{Aed Abschaetzung}
\E{\int_0^T \nos{\Aed \Yc(s)}_{\Hm} \d s} \leq C.
\end{align}
The estimates in (\ref{Constant interpolation estimate}) imply for fixed  $n \in \N$, $\eps,\,\delta>0$
the existence of a subsequence, still denoted by $\{\Yc\}_{\tau,h>0}$,
and a $Y \in L^2(\Omega\times (0,T);\L)\cap L^2(\Omega\times (0,T);\Hz)\cap L^{\infty}((0,T);L^2(\Omega;\L)$, s.t., for $\tau,h \rightarrow 0$
\begin{align}\label{limit_process}
\Yc &\weak Y ~\text{in}~ L^2(\Omega\times (0,T);\L), \nonumber\\
\Yc &\weak Y ~\text{in}~ L^2(\Omega\times (0,T);\Hz),\\
\Yc &\weak^* Y ~\text{in}~ L^{\infty}((0,T);L^2(\Omega;\L)) \nonumber.
\end{align}
In addition, there exists a $\nu \in L^2(\Omega;\L)$ such that $\Yc(T) \rightharpoonup \nu$ 
in $L^2(\Omega;\L)$ as $\tau,h \rightarrow 0$ and the estimate (\ref{Aed Abschaetzung}) 
implies the existence of a $a^{\epsilon,\delta} \in L^2(\Omega\times (0,T);\Hm)$, s.t.,
\begin{align}\label{lim_a}
\Aed \Yc &\weak a^{\eps,\delta} ~\text{in}~ L^2(\Omega\times (0,T);\Hm)\quad \mathrm{for} \,\,\tau,h \rightarrow 0.
\end{align} 
The estimates in (\ref{Constant interpolation estimate}) also implies for fixed  $n \in \N$, $\eps,\, \delta > 0$
the existence of a subsequence, still denoted by $\{\Ycm\}_{\tau>0}$,
and a $Y^- \in L^2(\Omega\times (0,T);\L)$, s.t.,
\begin{align}\label{limit_process_minus}
\Ycm &\weak Y^- ~\text{in}~ L^2(\Omega\times (0,T);\L)\quad \mathrm{for} \,\,\tau,h \rightarrow 0.
\end{align}
Finally, the inequality \eqref{num.visc.energie.estimate3} implies
\begin{align}\label{same_weak}
\lim\limits_{\tau \rightarrow 0}\E{\int_0^T \nos{\Yc(s)-\Ycm(s)} \d s }=&\lim\limits_{\tau \rightarrow 0}\tau \E{\sum_{k=1}^N\nos{X_{\eps,\delta,n}^k-X_{\eps,\delta,n}^{k-1}}} \nonumber\\
\leq& \lim\limits_{\tau \rightarrow 0} C\tau =0\,.  
\end{align}
which shows that the weak limits of $Y$ and $Y^-$ coincide. 

The following result shows that the limit $Y\equiv \Xdn$, i.e., 
that the numerical solution of scheme (\ref{num.visc.TVF}) converges to the unique variational solution of (\ref{vis.TVF}) for $\tau,h \rightarrow 0$.
Owing to the properties (\ref{Monotonicity}), (\ref{a_bnd})
the convergence proof follows standard arguments for the convergence of numerical approximations of monotone equations, see for instance \cite{gm_05}, \cite{em_sis_18},
and is therefore omitted.
We note that the convergence of the whole sequence $\{\Yc\}_{\tau,h>0}$ follows by the uniqueness of the variational solution.
\begin{lems}\label{lemma_Limiten_Gleichung}
Let $x_0 \in L^2(\Omega,\F_0;\L)$ and $g\in\L$ be given, let $\eps, \delta, \lambda>0$, {$n \in \N$} be fixed. 
Further, let $\Xdn$ be the unique variational solution of \eqref{vis.TVF} 
for $ x^n_0=\mathcal{P}_nx_0$, $g^n = \mathcal{P}_n g$
and $\Yc$, $\Ycm$ be the respective time-interpolant (\ref{eps_delta_interpol1}), (\ref{eps_delta_interpol2}) of the numerical solution $\{\Yi\}_{i=1}^N$ of \eqref{num.visc.TVF}.
Then $\Yc$, $\Ycm$ converge to $\Xdn$ for $\tau,h \rightarrow 0$ in the sense that
the weak limits from (\ref{limit_process}), (\ref{lim_a}) satisfy $Y\equiv \Xdn$, $a^{\epsilon,\delta}\equiv \Aed Y \equiv \Aed \Xdn$ and $\nu=Y(T)\equiv \Xdn(T)$.
In addition it holds for almost all $(\omega,t) \in \Omega\times (0,T)$ that
\begin{align*}
Y(t)=Y(0){-}\intt \Aed Y(s) \d s+\intt Y(s)\d W(s),
\end{align*}
and there is an $\L$-valued continuous modification of $Y$ (denoted again as $Y$) such that for all $t \in [0,T]$ 
\begin{align}\label{Ito-Formule_fuer_Limiten}
\frac{1}{2}\nos{Y(t)}= & \frac{1}{2}\nos{Y(0)}{-}\intt \langle \Aed Y(s),Y(s) \rangle +\frac{1}{2}\nos{Y(s)} \d s
\\ \nonumber 
& +\intt (Y(s),Y(s))\d W(s).
\end{align}
\end{lems}

{The strong monotonicity property \eqref{Monotonicity} of the operator $\Aed$ implies strong convergence of the numerical
approximation in  $L^2(\Omega\times(0,T);\L)$.}
\begin{lems}\label{Lemma_Convergence_num.vis.Scheme}
Let $x_0 \in L^2(\Omega,\F_0;\L)$ and $g\in\L$ be given,
let $\eps, \delta, \lambda>0$, {$n \in \N$} be fixed. Further, let $\Xdn$ be the variational solution of \eqref{vis.TVF} 
for $ x^n_0=\mathcal{P}_nx_0$, $g^n = \mathcal{P}_n g$
and $\Yc$ be the time-interpolant (\ref{eps_delta_interpol1}) of the numerical solution $\{\Yi\}_{i=1}^N$ of \eqref{num.visc.TVF}.
Then the following convergence holds true
\begin{align}
\lim\limits_{\tau,h \rightarrow 0}\nos{\Xed-\Yc}_{L^2(\Omega \times (0,T);\L)}\rightarrow 0.
\end{align}
\end{lems}
\begin{proof}[\textbf{Proof of Lemma \ref{Lemma_Convergence_num.vis.Scheme}}]
The proof follows along the lines of \cite{gm_05}, \cite{em_sis_18}.
We sketch the main steps of the proof for the convenience of the reader.

We note that $\Yc$ satisfies (cf. proof of Lemma~\ref{Lemma_Discrete a priori estimates})
\begin{align}\label{psi estimate}
e^{-\kappa T} \E{\nos{\Yc(T)}} \leq& \E{\nos{x_0^n}} -\kappa\int_0^T e^{-\kappa s} \E{\nos{\Yc(s)}}\d s \nonumber\\ 
-& 2\E{\int_0^T e^{-\kappa s} \langle \Aed \Yc(s),\Yc(s)\rangle  \d s} \\ 
+&\E{\int_0^T e^{-\kappa s} \nos{\Yc(s)} \d s} 
   + \kappa \int_0^T e^{-\kappa s}\be{R_\tau(s)}\d s \nonumber,
\end{align}
where $\displaystyle R_{\tau}(t):= \E{\int_t^{\theta_+(t)}2\langle\Aed \Yc(s),\Yc(s)\rangle-\nos{\Yc(s)}\d s }$.

We reformulate the third term on the right-hand side in \eqref{psi estimate} as
\begin{align*}
&\E{\int_0^T e^{-\kappa s} \langle \Aed \Yc(s),\Yc(s)\rangle \d s}\\
&= \E{\int_0^T e^{-\kappa s} \langle \Aed \Yc(s)-\Aed \Xed(s),\Yc(s)-\Xed(s)\rangle \d s}\\
&+\E{\int_0^T e^{-\kappa s} \langle \Aed \Xed(s),\Yc(s)-\Xed(s)\rangle+\langle\Aed \Yc(s),\Xed(s)\rangle \d s}.
\end{align*}
We substitute the equality above into \eqref{psi estimate} and obtain for $\kappa \geq 1$ that
\begin{align*}
& e^{-\kappa T} \E{\nos{\Yc(T)}}  +2\E{\int_0^T e^{-\kappa s} \langle \Aed \Yc(s)-\Aed \Xed(s),\Yc(s)-\Xe(s)\rangle \d s}
\\
& \leq \E{\nos{x_0^n}} 
- 2\E{\int_0^T e^{-\kappa s} \langle \Aed \Xed(s),\Yc(s)-\Xed(s)\rangle+\langle\Aed \Yc(s),\Xed(s)\rangle \d s}
\\
 &\qquad + \kappa\int_0^T e^{-\kappa s}\be{R_\tau(s)}\d s.
\end{align*}
We observe that $\displaystyle \int_0^T e^{-\kappa s}\be{R_\tau(s)}\d s \rightarrow 0$ for $\tau$.
Hence, by the lower-semicontinuity of norms 
using the convergence properties from Lemma~\ref{lemma_Limiten_Gleichung} 
and the monotonicity property \eqref{Monotonicity}
we get for $\tau,h \rightarrow 0$ that 
\begin{align}\label{inequality 3}
& e^{-\kappa T} \E{\nos{\Xed (T)}}  +2\lambda\lim\limits_{\tau,h\rightarrow 0} \E{\int_0^T e^{-\kappa s} \nos{\Yc(s)-\Xed(s)} \d s} \nonumber
\\
& \qquad + 2\delta\lim\limits_{\tau,h\rightarrow 0} \E{\int_0^T e^{-\kappa s} \nos{\nabla\big(\Yc(s)-\Xed(s)\big)} \d s}
\\ \nonumber
& \leq \E{\nos{x_0^n}} - 2\E{\int_0^T e^{-\kappa s} \langle \Aed \Xed(s),\Xed(s)\rangle \d s} .
\end{align}
It is not difficult to see that \eqref{Ito-Formule_fuer_Limiten} for $Y\equiv \Xed$ implies
\begin{align} \label{Equalitiy}
& e^{-\kappa T} \E{\nos{\Xed(T)}} = \E{\nos{x_0^n}}  
- 2\E{\int_0^T e^{-\kappa s} \langle \Aed(s)\Xed,\Xed(s)\rangle \d s }\\
& - \kappa\int_0^T e^{-\kappa s} \E{\nos{\Xed(s)}}\d s +\E{\int_0^T e^{-\kappa s}\nos{\Xed(s)} \d s} \nonumber.
\end{align}
We subtract the equality \eqref{Equalitiy} from \eqref{inequality 3} and obtain for $\kappa \geq 1$
\begin{align*}
\lambda \lim\limits_{\tau,h\rightarrow \infty} \E{\int_0^T e^{-\kappa s} \nos{\Yc(s)-\Xed(s)} \d s} \leq  0.
\end{align*}
Hence, we conclude that $ \Yc \rightarrow \Xdn $ in $L^2(\Omega;L^2((0,T);\L)$.
\end{proof}
{\begin{bems}
It is obvious from the proof of Lemma~\ref{Lemma_Convergence_num.vis.Scheme}
that the strong convergence in  $L^2(\Omega\times(0,T);\L)$ 
remains valid for $\lambda=0$ due to (\ref{Monotonicity}) by the Poincar\'e inequality.
\end{bems}}

Next lemma guarantees the convergence of the numerical solution of scheme \eqref{num.visc.TVF}
to the numerical solution of scheme (\ref{num.reg.TVF}) for $\delta \rightarrow 0$.
\begin{lems}\label{Lemma_Difference-num.Schemes}
{Let $x_0 \in L^2(\Omega,\F_0;\L)$ and $g\in\L$ be given.
Then for each $n \in \mathbb{N}$ there exists a constant $C\equiv C(T)>0$, $C_n\equiv C(\mathbb{E}[\|x_0^n\|_{\Hz}], \|g^n\|_{\Hz})>0$ }
such that for any $N\in\N$, $\delta>0$, $n\in \mathbb{N}$, $h,\eps \in (0,1]$ the following estimate holds for the difference of numerical solutions of (\ref{num.reg.TVF}) and (\ref{num.visc.TVF}):
\begin{align*}
\max\limits_{i=1,\ldots,N}\E{\nos{\Xi-\Yi}} \leq C(C_{n}\delta+ \E{\nos{x_0^h-x_0^{h,n}}}+ \lambda\nos{g^h-g^{h,n}}). 
\end{align*} 
\end{lems}
We note that the $n$-dependent constant $C_n$ in the estimate above is due to the a priori estimate
(\ref{discrete_H1_estimate_viscTVF}), for $\Hz$-regular data $x_0$, $g$ it holds that $C_n\equiv C(\mathbb{E}[\|x_0\|_{\Hz}], \|g\|_{\Hz})$
by the stability of the discrete $\L$-projection $\mathcal{P}_h:\Hz\rightarrow \mathbb{V}_h$ in $\Hz$.
\begin{proof}[\textbf{Proof of Lemma \ref{Lemma_Difference-num.Schemes}}]
We  define $\Zi:=\Xi -\Yi$. From \eqref{num.reg.TVF} and \eqref{num.visc.TVF} we get
\begin{align*}
\ska{\Zi,\vh}=& \ska{\Zmin.\vh}-\tau \delta \ska{\Delta_h \Yi,\vh}\\&-\tau \ska{\fe{\Xi},\vh }-\tau \ska{\fe{\Yi},\nabla\vh }\\&-\tau\lambda\ska{\Zi,\vh}-{\tau\lambda\ska{g^h - g^{h,n},\vh}}
\\
&+\ska{\Zmin,\vh}\Delta_i W.
\end{align*}
We set $\vh=\Zi$ and obtain
\begin{align*}
\ska{\Zi-\Zmin,\Zi} =&  -\tau \delta \ska{\Delta_h \Yi,\Zi}\\&-\tau \ska{\fe{\Xi}- \fe{\Yi},\nabla\Zi }\\&-\tau\lambda\nos{\Zi}-\tau\lambda\ska{g^h- g^{h,n},\Zi}
\\
&+\ska{\Zmin,\Zi}\Delta_i W.
\end{align*}
We note that 
\begin{align*}
\ska{\Zi-\Zmin,\Zmin}= \frac{1}{2}\nos{\Zi}-\frac{1}{2}\nos{\Zmin}+\frac{1}{2}\nos{\Zi -\Zmin}\,,
\end{align*}
and by the Cauchy-Schwarz and Young's inequalities

\begin{align*}
&\tau \delta \ska{\Delta_h \Yi,\Zi}\leq  \frac{\tau \delta^2}{2\lambda } \nos{\Delta_h \Yi} +\frac{\tau \lambda}{2}\nos{\Zi},\\
  &\tau \lambda \ska{g^h-g^{h,n} ,\Zi}\leq  \frac{\tau \lambda}{2} \nos{g^h-g^{h,n}} +\frac{\tau \lambda}{2}\nos{\Zi}.
\end{align*}

From the convexity (\ref{eps.convexity.inequality}) it follows that
\begin{align*}
-\tau \ska{\fe{\Xi}-\fe{\Yi},\nabla(\Xi-\Yi)}\leq 0.
\end{align*}
Hence, we obtain that
\begin{align}\label{zest1}
&\frac{1}{2}\nos{\Zi}+\frac{1}{2}\nos{\Zi -\Zmin}\\
&\leq \frac{1}{2}\nos{\Zmin}+\frac{\tau \delta^2}{2\lambda } \nos{\Delta_h \Yi}+ \frac{\tau \lambda}{2} \nos{g^h-g^{h,n}} +\ska{\Zmin,\Zi}\Delta_i W\,. \nonumber
\end{align}
We estimate the last term on the right-hand side above as
\begin{align*}
\ska{\Zmin,\Zi}\Delta_i W=\ska{\Zmin,\Zi-\Zmin}\Delta_i W+\nos{\Zmin}\Delta_i W\\
\leq \frac{1}{2}\nos{\Zi-\Zmin}+\frac{1}{2}\nos{\Zmin}\bes{\Delta_i W} +\nos{\Zmin}\Delta_i W,
\end{align*}
and substitute the above identity into (\ref{zest1}) 
\begin{align*}
\frac{1}{2}\nos{\Zi}+\frac{1}{2}\nos{\Zi -\Zmin}\leq& \frac{1}{2}\nos{\Zmin}+\frac{\tau \delta^2}{2\lambda } \nos{\Delta_h \Yi}+ \frac{\tau \lambda}{2} \nos{g^h-g^{h,n}} +\frac{1}{2}\nos{\Zi-\Zmin}\\
&+\frac{1}{2}\nos{\Zmin}\bes{\Delta_i W}+\nos{\Zmin}\Delta_i W\,.
\end{align*}
Next, we sum up the above inequality up to $i\leq N$ and obtain
\begin{align*}
\frac{1}{2}\nos{\Zi} \leq& \frac{1}{2}\nos{Z^0_{\epsilon,h}}+\frac{\tau \delta^2}{2\lambda } \sum\limits_{k=1}^{i} \nos{\Delta_h X_{\epsilon,\delta,h}^k}
+\frac{1}{2}\sum\limits_{k=1}^{i}\nos{Z^{k-1}_{\epsilon,h}}\bes{\Delta_k W}+\sum\limits_{k=1}^{i}\nos{Z^{k-1}_{\epsilon,h}}\Delta_k W\\&+ \frac{T \lambda}{2} \nos{g^h- g^{h,n}}.
\end{align*}
After taking expectation in the above and using the independence properties of Wiener increments
and the estimate (\ref{discrete_H1_estimate_viscTVF}) we arrive at
\begin{align*}
\frac{1}{2}\E{\nos{\Zi}} \leq &\frac{1}{2}\nos{Z^0_{\epsilon,h}}+\frac{\tau \delta^2}{2\lambda } \E{\sum\limits_{k=1}^{i} \nos{\Delta_h X_{\epsilon,\delta,n,h}^k}}
+\frac{\tau}{2}\sum\limits_{k=1}^{i}\E{\nos{Z^{k-1}_{\epsilon,h}}}\\
\leq&C_n \delta+\frac{1}{2}\E{\nos{Z^0_{\epsilon,h}}} +\frac{\tau}{2}\sum\limits_{k=0}^{i-1}\E{\nos{Z^k_{\epsilon,h}}}+ \frac{T \lambda}{2} \nos{g^h-g^{h,n}}.
\end{align*}
with $C_n \equiv C(\|x_0^n\|_{\Hz}, \|g^n\|_{\Hz})$.
Finally, the Discrete Gronwall lemma yields for $i=1,\ldots,N$ that
\begin{align}\label{m,n gronwall}
\E{\nos{\Zi}} \leq \exp(T)(C_{n}\delta+ \frac{1}{2}\E{\nos{x_0^h- x_0^{h,n}}}+ \frac{T \lambda}{2}\nos{g^h-g^{h,n}}).
\end{align}
which concludes the proof .
\end{proof}
We define piecewise constant time-interpolant
of the discrete  solution $\{\Xi \}_{i=0}^N$ of (\ref{num.reg.TVF}) for $t\in[0,T)$ as
\begin{align}\label{eps_interpol}
\Xc(t) = \Xi\quad \mathrm{if}\quad t \in (t_{i-1},t_i]. 
\end{align}
We are now ready to state the second main result of this paper
which is the convergence of the numerical approximation (\ref{num.reg.TVF}) to the unique SVI solution of the total variation flow (\ref{TVF})
(cf. Definition~\ref{def_varsoleps}).
\begin{thms}\label{Thm_Convergence_num.reg.Scheme}
Let $X$ be the SVI solution of \eqref{TVF} and let $\Xc$ be the time-interpolant (\ref{eps_interpol})
of the numerical solution of the scheme \eqref{num.reg.TVF}.
Then the following convergence holds true
\begin{align}\label{numconv}
\lim\limits_{\epsilon\rightarrow 0}\lim\limits_{\tau,h \rightarrow 0}\nos{X-\Xc}_{L^2(\Omega \times (0,T);\L)}\rightarrow 0.
\end{align}
\end{thms}
\begin{proof}[\textbf{Proof of Theorem \ref{Thm_Convergence_num.reg.Scheme}}]
For $x_0 \in L^2(\Omega,\F_0;\L)$ and $g\in\L$  
we define the $\Hz$-approximating sequences $\{x_0^n\}_{n\in \mathbb{N}}\subset \mathbb{H}^1_0$, $x_0^n\rightarrow x_0\in L^2(\Omega,\F_0;\L)$,
$\{g^n\}_{n\in \mathbb{N}}\subset \mathbb{H}^1_0$, $n\in\mathbb{N}$, $g^n\rightarrow g\in\L$
via the $\L$-projection onto $\mathbb{V}_n\subset\Hz$. 
We consider the solutions  $\Xe,\Xdn$  of \eqref{reg.TVF}, \eqref{vis.TVF}, respectively,
and denote by $\Xe_n$ the SVI solution of \eqref{reg.TVF} for $x_0\equiv x_0^n$, $g\equiv g^n$.
Furthermore, we recall that the interpolant $\Yc$ of the numerical solution of (\ref{num.visc.TVF}) was defined in (\ref{eps_delta_interpol1}).

We split the numerical error as
{\begin{align}\label{err_ineq}
\nonumber
\frac{1}{5}\nos{X-\Xc}_{L^2(\Omega \times (0,T);\L)} \leq& \nos{X-\Xe}_{L^2(\Omega \times (0,T);\L)} +\nos{\Xe-\Xe_n}_{L^2(\Omega \times (0,T);\L)}
\\
&+ \nos{\Xe_n-\Xdn}_{L^2(\Omega \times (0,T);\L)}+\nos{\Xdn-\Yc}_{L^2(\Omega \times (0,T);\L)}
\\
\nonumber
&+\nos{\Yc-\Xc}_{L^2(\Omega \times (0,T);\L)}
\\
\nonumber
& =:I+II+III+IV+V.
\end{align}
}
By Theorem  \ref{Thm_reg.SVI} it follows that
\begin{align*}
 \lim\limits_{\epsilon \rightarrow 0}\, I = \lim\limits_{\epsilon \rightarrow 0} \nos{X-\Xe}_{L^2(\Omega \times (0,T);\L)}=0.
\end{align*}
To estimate the second term we consider the solutions $\Xe_n$ of (\ref{reg.TVF}) with $x_0\equiv x_0^n$ and $g\equiv g^n$.
From  (\ref{reg_stability_inequality}) we deduce that
{
\begin{align*}
 \lim\limits_{n \rightarrow \infty}\, II  = & \lim_{n\rightarrow \infty}
 \nos{\Xe-\Xe_n}_{L^2(\Omega \times (0,T);\L)}
=  0\, .
\end{align*}
}
We use (\ref{delta_limit})  to estimate the third term  as 
$$
\lim\limits_{\delta\rightarrow 0}III=\lim_{\delta\rightarrow 0}\nos{\Xe_n-\Xdn}_{L^2(\Omega \times (0,T);\L)} = 0\,.
$$
The fourth term is estimated by Lemma~\ref{Lemma_Convergence_num.vis.Scheme}
\begin{align*}
\lim\limits_{\tau,h\rightarrow 0}IV=\lim\limits_{\tau,h\rightarrow 0}\nos{\Xdn-\Yc}_{L^2(\Omega \times (0,T);\L)} =0.
\end{align*} 
For the last term we use Lemma \ref{Lemma_Difference-num.Schemes}
\begin{align*}
\lim\limits_{n \rightarrow \infty}\limsup_{\delta \rightarrow 0} V=\lim\limits_{n \rightarrow \infty}\limsup\limits_{\delta \rightarrow 0}\nos{\Yc-\Xc}_{L^2(\Omega \times (0,T);\L)}=0.
\end{align*}
Finally, we consecutively take  $\tau,h \rightarrow 0$,  $\delta\rightarrow 0$,
$n\rightarrow \infty$ and $\epsilon \rightarrow 0$ in (\ref{err_ineq})
and use the above convergence of $I-V$ to obtain (\ref{numconv}).
\end{proof}

\begin{bems}
We note that the convergence analysis simplifies in the case that the problem data have higher regularity. For $x_0,g\in\Hz$
it is possible to show that the problem \eqref{reg.TVF} admits a unique variational solution (which is also a SVI solution of \eqref{reg.TVF} by uniqueness)
by a slight modification of standard monotonicity arguments.
This is due to the fact that the operator (\ref{Operator}) retains all its properties for $\delta=0$
except for the coercivity. The coercivity is only required to guarantee $\Hz$-stability of the solution, nevertheless the stability
can also be obtained directly by the It\^o formula on the continuous level, cf. Lemma~\ref{laplace_energy_estimate}, or analogically to Lemma~\ref{Lemma_Discrete a priori estimates} on the discrete level, even for $\delta=0$.
Consequently, for $\Hz$-data the convergence of the numerical solution $\Xc$ can be shown as in Theorem~\ref{Thm_Convergence_num.reg.Scheme} without the additional $\delta$-regularization step.
\end{bems}

We conclude this section by showing unconditional stability of scheme \eqref{num.reg.TVF}, i.e.,
we show that the numerical solution satisfies a discrete energy law which is an analogue of the energy estimate (\ref{cont_ener}).
\begin{lems}\label{eps.num.energy.estimates}
Let $x_0, g\in\L$ and $T> 0$.
Then there exist a constant $C\equiv C(T)$ such that 
the solutions of scheme \eqref{num.reg.TVF} satisfy for any $\eps,h\in (0,1]$, $N\in \mathbb{N}$ 
\begin{align}\label{disc_ener}
 \sup_{i=1,\ldots,N}\frac{1}{2}\E{\nos{\Xi}}&+\tau \E{\sum\limits_{i=1}^N\mathcal{J}_{\epsilon}(\Xi) +\frac{\lambda}{2}\nos{\Xi-g^h}}\ \nonumber\\
&\leq C\left(\frac{1}{2}\nos{x_0}+ T\eps\be{\O}+\frac{T\lambda}{2}\nos{g})\right)\,.
\end{align}
\end{lems}
\begin{proof}[\textbf{Proof of Lemma \ref{eps.num.energy.estimates}}]
We set $v_h\equiv\Xi$ in \eqref{num.reg.TVF} and obtain
\begin{align}\label{1}
& \frac{1}{2}\nos{\Xi}+\frac{1}{2}\nos{\Xi-\Xmin}+\tau \ska {\fe{\Xi},\nabla \Xi}+\tau \lambda(\Xi-g^h,\Xi)
\nonumber \\
& \qquad =\frac{1}{2}\nos{\Xmin}+(\Xmin,\Xi)\Delta_i W.
\end{align}
Using the the convexity of $\mathcal{J}_\eps$ along with the inequality 
\begin{align*}
(\Xmin,\Xi)\Delta_i W & = (\Xmin,\Xi-\Xmin)\Delta_i W + \|\Xmin\|^2\Delta_i W
\\
& \leq
\frac{1}{2} \|\Xi-\Xmin\|^2 + \frac{1}{2}\|\Xmin\|^2|\Delta_i W|^2 + \|\Xmin\|^2\Delta_i W,
\end{align*}
we get from (\ref{1}) that
\begin{align}\label{2}
\frac{1}{2}\nos{\Xi}&+\frac{1}{2}\nos{\Xi-\Xmin}+\tau \mathcal{J}_{\epsilon}(\Xi) +\frac{\tau\lambda}{2}\nos{\Xi-g^h}\nonumber\\
\leq&\tau\mathcal{J}_{\epsilon}(0) +\frac{1}{2}\nos{\Xmin} 
+\frac{1}{2}\nos{\Xi-\Xmin} \\
&+\frac{1}{2}\nos{\Xmin}|\Delta_i W|^2+\nos{\Xmin}\Delta_i W.\nonumber 
\end{align}
After taking the expectation and summing up over $i$ in \eqref{2},
and noting that $\mathcal{J}_{\epsilon}(0) =\epsilon\be{\O}$ we obtain
\begin{align*}
\frac{1}{2}\E{\nos{\Xi}}&+\tau \E{\sum\limits_{k=1}^i\mathcal{J}_{\epsilon}(X^{k}_{\epsilon,h}) +\frac{\lambda}{2}\nos{X^{k}_{\epsilon,h}-g^h}} \nonumber\\
\leq &\frac{1}{2}\nos{x_0}+T\Big(\epsilon\be{\O}+\frac{\lambda}{2}\nos{g}\Big)+\frac{\tau}{2}\E{\sum\limits_{k=0}^{i-1}\nos{X^{k}_{\epsilon,h}}}.
\end{align*}
Hence (\ref{disc_ener}) follows after an application of the discrete Gronwall lemma.
\end{proof}
}

\section{Numerical experiments}\label{sec_sim}
We perform numerical experiments using a generalization of the fully discrete finite element {scheme} (\ref{num.reg.TVF})
on the unit square $\D= (0,1)^2$. 
The scheme for $i=1, \dots, N$ then reads as
\begin{align}\label{fem_scheme}
\ska{X^i_{\varepsilon,h},\vh} =& \ska{X^{i-1}_{\varepsilon,h},\vh}-\tau \ska{\fe{X^i_{\varepsilon,h}},\nabla\vh } \nonumber \\
 & -\tau\lambda\ska{X^i_{\varepsilon,h} - g^h,\vh}+ \mu \ska{\sigma(X^{i-1}_{\varepsilon,h}) {\Delta_i W^h},\vh} &&\forall \vh \in \mathbb{V}_h\,, 
\\ \nonumber
X^0_{\varepsilon,h}  = & x_0^h\,,
\end{align}
where $g^h,\,x_0^h\in\mathbb{V}_h$ are suitable approximations of $g$, $x_0$ (e.g., the orthogonal projections onto $\mathbb{V}_h$), respectively,
and $\mu>0$ is a constant. The multiplicative space-time noise $\sigma(X^{i-1}_{\varepsilon,h})\Delta_i W^h$ is constructed as follows.
The term $W^h$ is taken to be a $\mathbb{V}_h$-valued space-time noise of the form
$$
\Delta_i {W}^h(\xx)  =  \sum_{\ell=1}^{L} \phi_\ell (x) \Delta_{i} {\beta}_\ell
\qquad \forall\,  x \in \overline{\mathcal D}\,,
$$
where ${\beta}_\ell$, $\ell=1,\dots, L$ are independent scalar-valued Wiener processes and
$\{\phi_\ell\}_{\ell=1}^L$ is the standard 'nodal' finite element basis of $\mathbb{V}_h$.
In the simulations below we employ three practically relevant choices of $\sigma$:
a tracking-type noise $\sigma(X) \equiv \sigma_1(X) = |X - g^h|$,
a gradient type noise
$\sigma(X) \equiv \sigma_2(X) = |\nabla X |$ and the additive noise $\sigma(X)\equiv \sigma_3 = 1$;
in the first case the noise is small when the solution is close to the 'noisy image' $g^h$,
in the gradient noise case the noise is localized along the edges of the image.
We note that the fully discrete finite element scheme (\ref{fem_scheme}) corresponds to an approximation of the regularized equation (\ref{reg.TVF})
with a slightly more general space-time noise term of the form $\mu \sigma(\Xe) \mathrm{d}W$.

In all experiments we set $T=0.05$, $\lambda=200$,  $x_0 \equiv x_0^h\equiv 0$.
If not mentioned otherwise we use 
the time step $\tau = 10^{-5}$, the mesh size $h = 2^{-5}$ and set $\eps=h=2^{-5}$, $\mu=1$.
We define $g \in \mathbb{V}_h$ as a piecewise linear interpolation
of the characteristic function of a circle with radius $0.25$ on the finite element mesh, see Figure~\ref{fig_data} (left), 
and set $g^h = g + \xi_h \in \mathbb{V}_h$ with $\displaystyle \xi_h(x) = \nu \sum_{\ell=1}^{L} \phi_\ell (x) \xi_\ell$, $x\in\D$
where $\xi_\ell$, $\ell=1,\dots, L$ are realizations of independent $\mathcal{U}(-1,1)$-distributed random variables.
If not indicated otherwise we use $\nu=0.1$; the corresponding realization of $\xi_h$ is displayed in Figure~\ref{fig_data} (right).
\begin{figure}[!htp]
\center
\includegraphics[width=0.3\textwidth]{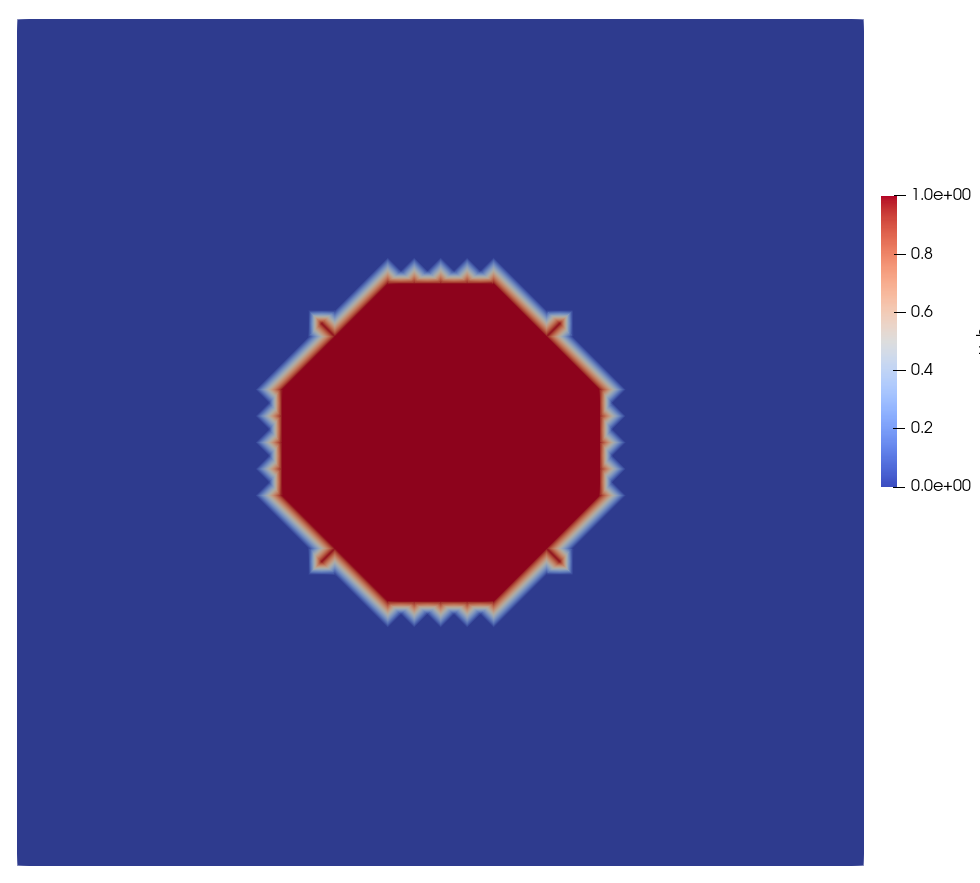}
\includegraphics[width=0.3\textwidth]{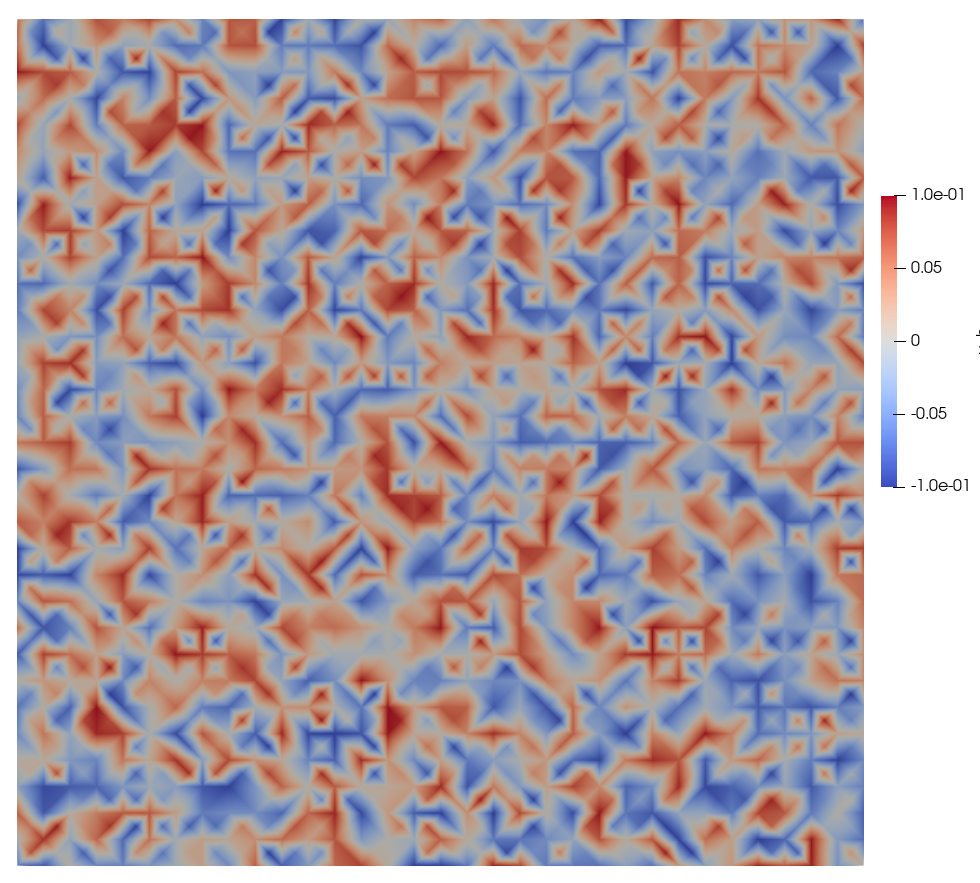}
\includegraphics[width=0.3\textwidth]{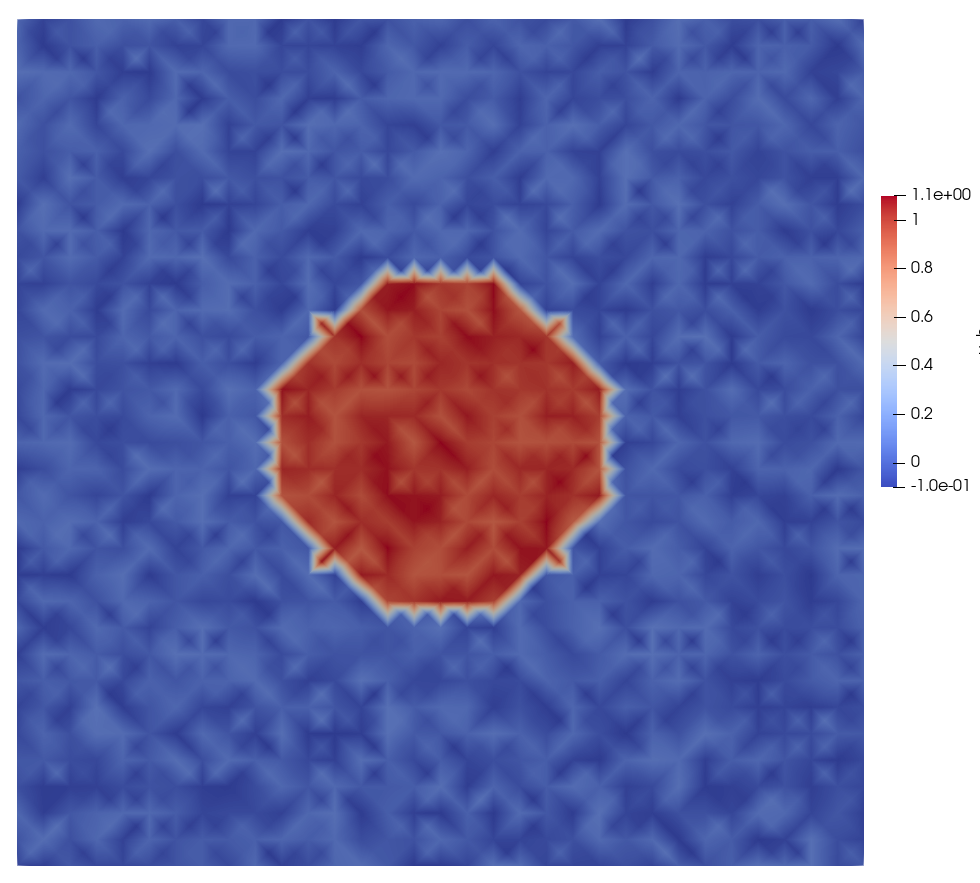}
\caption{The function $g$ (left), the noise $\xi_h$ (middle) and the noisy image (right).}
\label{fig_data}
\end{figure}

We choose $\eps=h=2^{-5}$, $\mu=1$, $\sigma\equiv \sigma_1$ as parameters for the 'baseline' experiment;
the individual parameters are then varied in order to demonstrate their influence on the evolution.
The time-evolution of the discrete energy functional $\mathcal{J}_{\eps,\lambda}(\Xi)$, $i=1,\dots,N$
for a typical realization of the space-time noise $W^h$ is displayed in Figure~\ref{fig_ener}; in the legend of the graph we state parameters
which differ from the parameters of the baseline experiment, e.g., the legend '$sigma_2,\, mu=0.125$'
corresponds to the parameters $\sigma\equiv \sigma_2$, $\mu=0.125$ and the remaining parameters are left unchanged, i.e., $\eps=h=2^{-5}$. 
For all considered parameter setups, except for the case of noisier image $\nu=0.2$, 
the evolution remained close to the discrete energy of the deterministic problem (i.e., (\ref{fem_scheme}) with $\mu=0$).
The energy decreases over time until the solution is close to the (discrete) minimum of $\mathcal{J}_{\eps,\lambda}$;
to highlight the differences we display a zoom at the graphs.
We observe that in the early stages (not displayed) the energy of stochastic evolutions with sufficiently small noise typically remained
below the energy of the deterministic problems and the situation reversed as the solution approached the stationary state.
\begin{figure}[!htp]
\center
\includegraphics[width=0.3\textwidth]{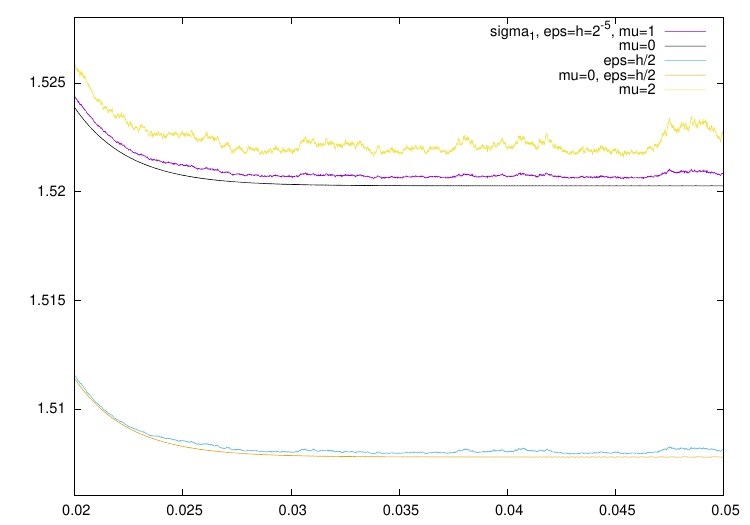}
\includegraphics[width=0.3\textwidth]{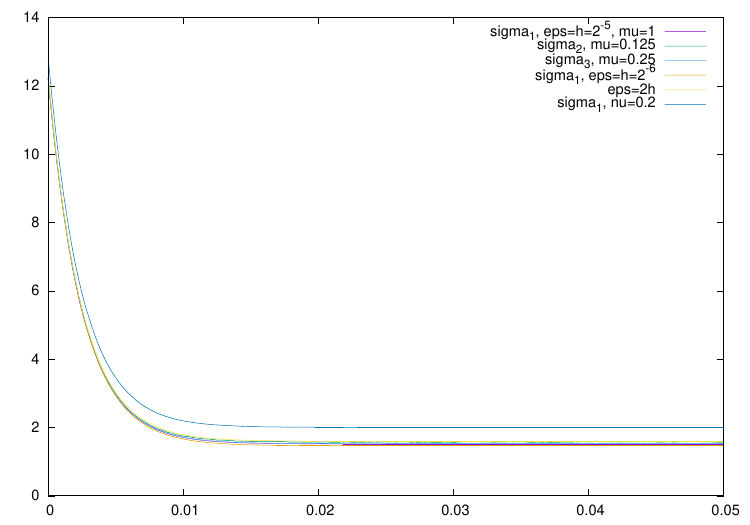}
\includegraphics[width=0.3\textwidth]{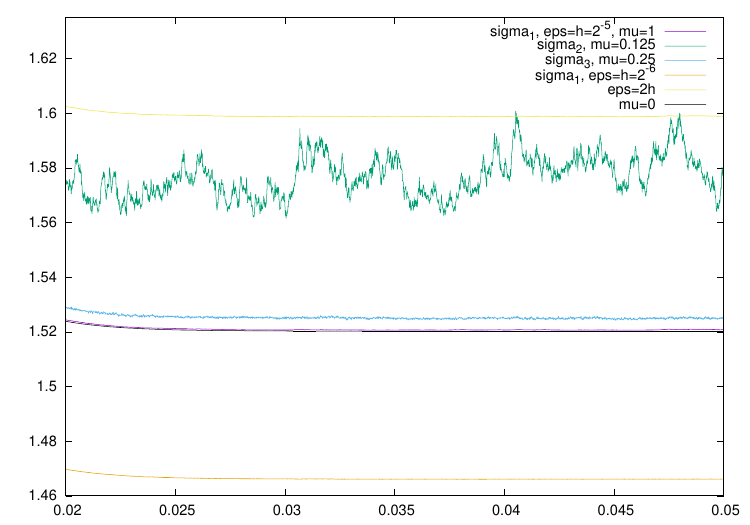}
\caption{Evolution of the discrete energy: $\sigma\equiv\sigma_1$, $h=2^{-5}$, $\eps=h,\frac{h}{2}$, $\mu=1,2$ (left); 
$\sigma\equiv\sigma_1,\sigma_2,\sigma_3$, $\sigma\equiv\sigma_1$, $h=2^{-5}$, $\eps=2h$, $\sigma=\sigma_1$, $\eps=h=2^{-6}$ and $\nu=0.2$ (middle and right).
}
\label{fig_ener}
\end{figure}

In Figure~\ref{fig_finalu} we display the solution at the final time
computed with $\sigma \equiv \sigma_1$, $\eps=h$ for $h=2^{-5}, 2^{-6}$, respectively,
and $\sigma \equiv \sigma_2$, $\eps=h=2^{-5}$; graphically the results of the remaining simulations 
did not significantly differ from the first case. 
The displayed results may indicate that the noise $\sigma_2$ yields worse results than the noise $\sigma_1$ and $\sigma_2$;
however, for sufficiently small value of $\mu$ the results would remain close to the deterministic simulation as well.
We have magnified noise intensity $\mu$ to 
highlight the differences to the other noise types (i.e., the noise is concentrated along the edges of the image).
We note that the gradient type noise $\sigma_2$ might be a preferred choice for practical computations, cf. \cite{swp14}.
\begin{figure}[!htp]
\center
\includegraphics[width=0.3\textwidth]{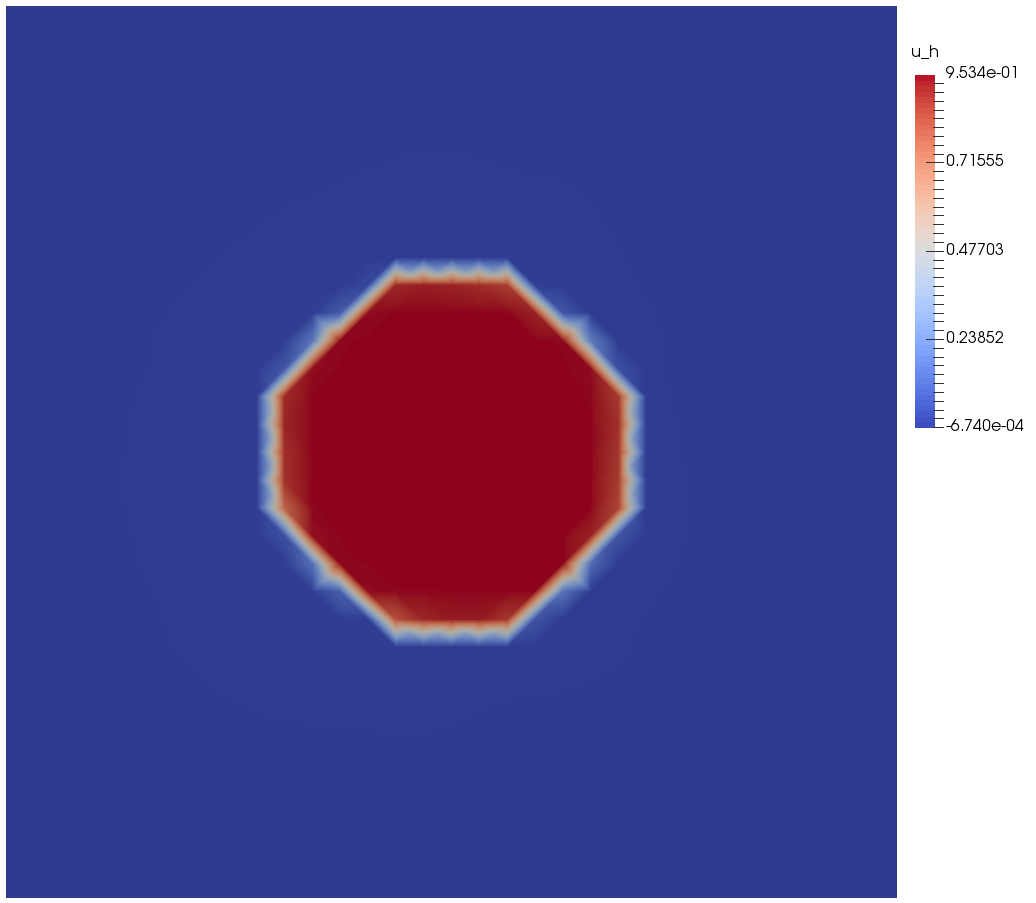}
\includegraphics[width=0.3\textwidth]{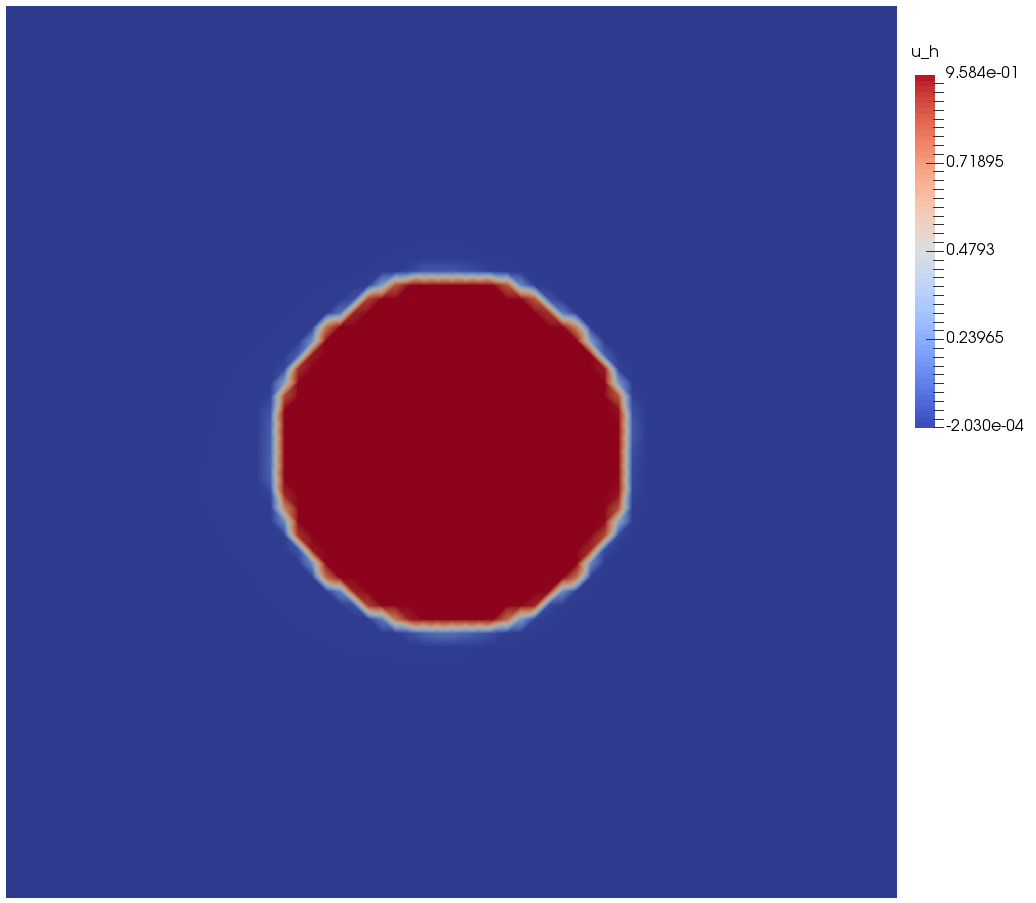}
\includegraphics[width=0.3\textwidth]{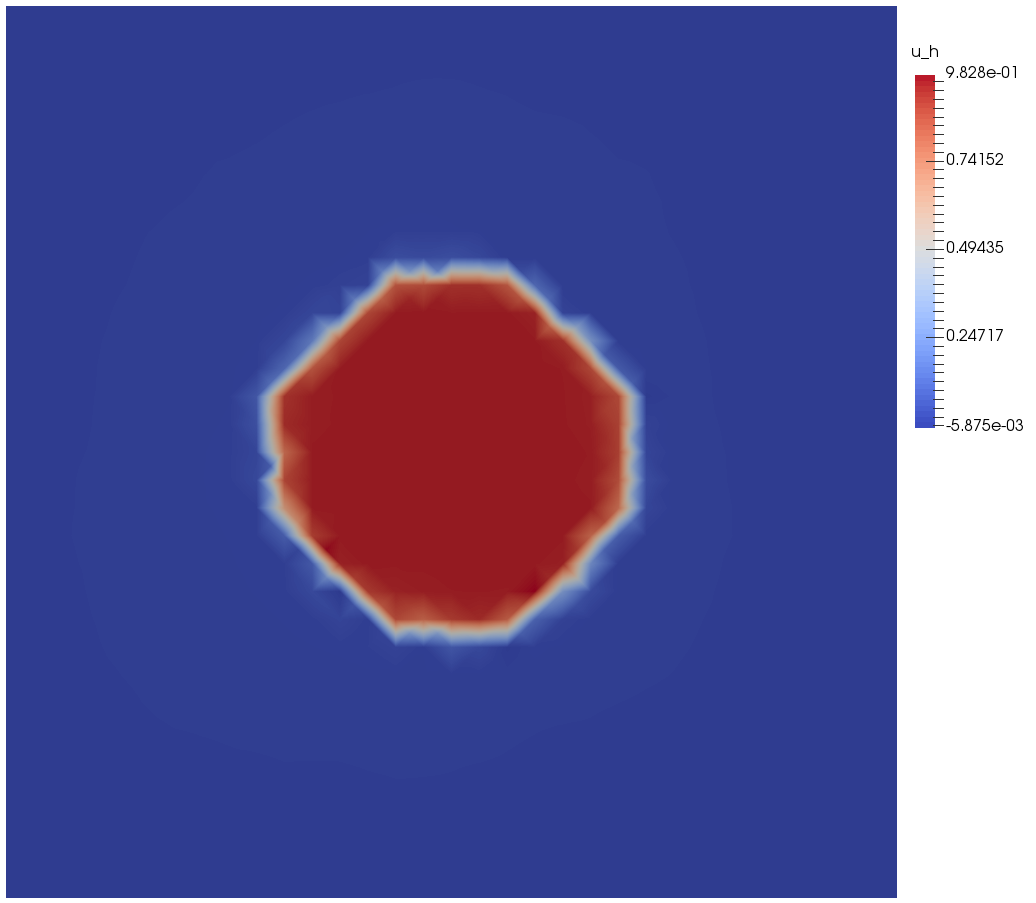}
\caption{From let to right: 
the solution for $\sigma \equiv \sigma_1$ with $\eps= h=2^{-5}$,
$\sigma \equiv \sigma_1$ with  $\eps= h =2^{-6}$ and
$\sigma \equiv \sigma_2$ with $\eps=2^{-5}$.}
\label{fig_finalu}
\end{figure}

\section*{Acknowledgement}
This work was supported by the Deutsche Forschungsgemeinschaft through SFB  1283  “Taming  uncertainty  and  profiting  from  randomness  and  low  regularity  in  analysis, stochastics and their applications”. 
The authors would like to thank the referee for careful reading of the manuscript and constructive comments, as well as to Lars Diening for stimulating discussions.
We would also like to thank Martin Ondrej\'at for pointing to us inaccuracies in the proof of uniqueness in Theorem~\ref{Thm_reg.SVI}.

\bibliographystyle{plain}
\bibliography{refs_short}

\begin{thebibliography}{10}

\bibitem{Ambrosio_functions_of_BV}
L.~Ambrosio, N.~Fusco, and D.~Pallara.
\newblock {\em Functions of bounded variation and free discontinuity problems}.
\newblock Oxford Mathematical Monographs. The Clarendon Press, Oxford
  University Press, New York, 2000.

\bibitem{book_attouch}
H.~Attouch, G.~Buttazzo, and G.~Michaille.
\newblock {\em Variational analysis in {S}obolev and {BV} spaces}, volume~6 of
  {\em MPS/SIAM Series on Optimization}.
\newblock Society for Industrial and Applied Mathematics (SIAM), Philadelphia,
  PA; Mathematical Programming Society (MPS), Philadelphia, PA, 2006.
\newblock Applications to PDEs and optimization.

\bibitem{Roeckner_TVF_paper}
V.~Barbu and M.~R\"{o}ckner.
\newblock Stochastic variational inequalities and applications to the total
  variation flow perturbed by linear multiplicative noise.
\newblock {\em Arch. Ration. Mech. Anal.}, 209(3):797--834, 2013.

\bibitem{bm16}
S.~Bartels and M.~Milicevic.
\newblock Stability and experimental comparison of prototypical iterative
  schemes for total variation regularized problems.
\newblock {\em Comput. Methods Appl. Math.}, 16(3):361--388, 2016.

\bibitem{stvf_hd}
\v{L}. Ba\v{n}as, M.~R\"ockner, and A.~Wilke.
\newblock Convergent numerical approximation of the stochastic total variation
  flow with linear multiplicative noise: the higher dimensional case, 2022.
\newblock \url{https://arxiv.org/abs/2211.04162}.

\bibitem{stvf_erratum}
\v{L}. Ba\v{n}as, M.~R\"ockner, and A.~Wilke.
\newblock Correction to: Convergent numerical approximation of the stochastic
  total variation flow.
\newblock {\em Stoch PDE: Anal. Comp.}, 2022.
\newblock \url{https://doi.org/10.1007/s40072-022-00267-5}.

\bibitem{BrennerS02}
S.~C. Brenner and L.~R. Scott.
\newblock {\em The Mathematical Theory of Finite Element Methods (second
  edition)}.
\newblock Springer-Verlag, New York, 2002.

\bibitem{em_sis_18}
E.~Emmrich and D.~\v{S}i\v{s}ka.
\newblock Nonlinear stochastic evolution equations of second order with
  damping.
\newblock {\em Stoch. Partial Differ. Equ. Anal. Comput.}, 5(1):81--112, 2017.

\bibitem{Prohl_TVF_numerics}
X.~Feng and A.~Prohl.
\newblock Analysis of total variation flow and its finite element
  approximations.
\newblock {\em M2AN Math. Model. Numer. Anal.}, 37(3):533--556, 2003.

\bibitem{Gess_Stability}
B.~Gess and J.~T\"{o}lle.
\newblock Stability of solutions to stochastic partial differential equations.
\newblock {\em J. Differential Equations}, 260(6):4973--5025, 2016.

\bibitem{gm_05}
I.~Gy\"{o}ngy and A.~Millet.
\newblock On discretization schemes for stochastic evolution equations.
\newblock {\em Potential Anal.}, 23(2):99--134, 2005.

\bibitem{kp06}
O.~Juan, R.~Keriven, and G.~Postelnicu.
\newblock Stochastic motion and the level set method in computer vision:
  Stochastic active contours.
\newblock {\em International Journal of Computer Vision}, 69:7--25, 2006.

\bibitem{kr_07}
N.V. Krylov and B.L. Rozovskii.
\newblock Stochastic evolution equations.
\newblock In {\em Stochastic differential equations: theory and applications},
  volume~2 of {\em Interdiscip. Math. Sci.}, pages 1--69. World Sci. Publ.,
  Hackensack, NJ, 2007.

\bibitem{Roeckner_book}
W.~Liu and M.~R\"{o}ckner.
\newblock {\em Stochastic partial differential equations: an introduction}.
\newblock Universitext. Springer, Cham, 2015.

\bibitem{ROF}
L.I. Rudin, S.~Osher, and E.~Fatemi.
\newblock Nonlinear total variation based noise removal algorithms.
\newblock {\em Phys. D}, 60(1-4):259--268, 1992.

\bibitem{swp14}
B.~Sixou, L.~Wang, and F.~Peyrin.
\newblock Stochastic diffusion equation with singular diffusivity and
  gradient-dependent noise in binary tomography.
\newblock {\em J. Phys.: Conf. Ser.}, 69:012001, 542.

\bibitem{book_temam}
R.~Temam.
\newblock {\em Navier-{S}tokes equations. {T}heory and numerical analysis}.
\newblock North-Holland Publishing Co., Amsterdam-New York-Oxford, 1977.
\newblock Studies in Mathematics and its Applications, Vol. 2.

\end{thebibliography}
\end{document}